\documentclass[a4paper,11pt]{article}
\usepackage{amsmath, amsthm, amsfonts, amssymb}
\numberwithin{equation}{section}

\newcommand{\R}{\mathbb R}
\newtheorem{theorem}{Theorem}[section]

\newtheorem{corollary}[theorem]{Corollary}

\newtheorem{definition}[theorem]{Definition}
\newtheorem{remark}[theorem]{\it \rmfamily Remark}

\newtheorem{example}[theorem]{Example}

\newtheorem{lemma}[theorem]{Lemma}

\newtheorem{hypothesis}{\it \rmfamily Hypothesis}

\newtheorem{proposition}[theorem]{Proposition}

\newcommand{\bth}{\begin{theorem}}
\def\eth{\end{theorem}}
\newcommand{\bpr}{\begin{proposition}}
\newcommand{\epr}{\end{proposition}}
\newcommand{\bco}{\begin{corollary}}
\newcommand{\eco}{\end{corollary}}
\newcommand{\ble}{\begin{lemma}}
\newcommand{\ele}{\end{lemma}}
\newcommand{\bpf}{\begin{proof}}
\newcommand{\epf}{\end{proof}}
\newcommand{\bex}{\begin{example}}
\newcommand{\eex}{\end{example}}
\newcommand{\bdf}{\begin{definition}}
\newcommand{\edf}{\end{definition}}
\newcommand{\bre}{\begin{remark}}
\newcommand{\ere}{\end{remark}}

\newcommand{\beq}{\begin{equation}}
\newcommand{\eeq}{\end{equation}}
\newcommand{\bal}{\begin{aligned}}
\newcommand{\eal}{\end{aligned}}
\newcommand{\ben}{\begin{enumerate}}
\newcommand{\een}{\end{enumerate}}
\newcommand{\beqr}{\begin{eqnarray*}}
\newcommand{\eeqr}{\end{eqnarray*}}

\def\R{{\mathbb R}}
\def\S{{\mathbb S}}

\def\lan{{\langle}}
\def\ran{{\rangle}}
\def\hh{{\vskip 0.5 mm \noindent }}

\def\qed{\hfill \hbox{\hskip 6pt\vrule
width6pt height6pt depth1pt  \hskip1pt}
\smallskip}

\begin{document}

\title{ \Large \bf PATHWISE UNIQUENESS FOR SINGULAR SDEs
 DRIVEN BY STABLE PROCESSES
\thanks{  Supported by the M.I.U.R. research project Prin 2008
``Deterministic and stochastic methods in the study of evolution
problems''.
% and by the Isaac Newton Institute for Mathematical %Sciences
% (program   SPD 85).
 } }

%\author{E. Priola
%\\ \\ {\small  Dipartimento di Matematica,
%Universit\`a di Torino,} \\ {\small
%via Carlo Alberto 10,   10123, Torino, Italy }\\
%{\small  e-mail  enrico.priola@unito.it   } }

\author{E. Priola }

\date{}

\maketitle

\begin{center}
\textbf{Abstract}
\end{center}
 We prove
 pathwise uniqueness
 for stochastic differential
 equations
  driven by  non-degenerate
  symmetric $\alpha$-stable  L\'evy processes with values in $\R^d$
   having a bounded and
   $\beta$-H\"older continuous  drift term. We assume
  $\displaystyle{ \beta > 1 - {\alpha}/{2} }$ and $\alpha \in [ 1, 2)$.
The proof  requires analytic regularity results for
  associated  integro-differential operators
   of Kolmogorov type.
 We also study differentiability
 of solutions with  respect to initial conditions
and the homeomorphism  property.
 %  \end{abstract}

%60H10   View Publications   (1973-now)  Stochastic %ordinary differential equations [See also 34F05]
%
%60J75   View Publications   (1973-now)  Jump %processes
%34F05   View Publications   (1973-now)  Equations and %systems with randomness
%35B65   View Publications   (1980-now)  Smoothness and %regularity of solutions

%{\vskip 5 mm }
% \noindent \textbf{MSC (2000)}
 %{\bf Mathematics  Subject Classification (2000):}
% 60H10, 35B65, 60J75,  34F05.

%\vspace{2.5 mm} \noindent {\bf Key words:} stochastic differential
%equations, stable processes, pathwise uniqueness, H\"older
%continuity.

\vspace{2.3 mm}

\section{Introduction}

In this paper we prove a pathwise uniqueness result for the
following
 SDE
\begin{equation} \label{SDE}
 X_{t}=x+\int_{0}^{t}b\left( X_{s}\right)  \mathrm{d}s+L_{t},
 \quad
  x \in \R^d, \;\; t \ge 0,
\end{equation}
 where $b: \R^d \to \R^d $ is bounded and
 $\beta$-H\"older continuous
 %for some  index $\beta \in (0,1)$,
  and $L=
 (L_t) $ is a non-degenerate
 $d$-dimensional symmetric  $\alpha$-stable L\'evy process,
  $d \ge 1$.

  Currently,   there is a great   interest in understanding
  pathwise uniqueness for SDEs when $b$ is not
  Lipschitz continuous or, more generally,
  when  $b$ is  singular enough so  that
  the corresponding  deterministic equation \eqref{SDE}
  with $L= 0$ is not well-posed.
   A remarkable result in this direction was
 proved by Veretennikov in \cite{Ver} (see also \cite{Zv74}
  for $d=1$).
 He was able to prove uniqueness
 when $b : \R^d \to \R^d$ is only Borel and bounded and
  $L $ is a standard $d$-dimensional Wiener process.
 This result has been  generalized in various directions
 in \cite{GM}, \cite{KR05}, \cite{Za}, \cite{D},
  \cite{FGP},
  \cite{DaPratoFlandoli}, \cite{FF}.

   The situation changes  when $L$ is not a Wiener process but is
  a symmetric
 $\alpha$-stable process, $\alpha \in (0,2)$.
  Indeed, when $d=1$ and $\alpha <1$, Tanaka,
  Tsuchiya and Watanabe prove in \cite[Theorem 3.2]{Tanaka}
  that even a bounded
  and $\beta$-H\"older continuous $b$ is not enough to ensure
  pathwise uniqueness  if
  $\alpha + \beta <1$ (they consider drifts like
   $b(x) = $sign$(x) \, (|x|^{\beta} \wedge 1)$ and initial
   condition
  $x=0$).
   On the other hand,
  when $d =1$ and $\alpha \ge 1$, they show pathwise
   uniqueness
  for any continuous and bounded $b$.

 In this paper we prove  pathwise
 uniqueness  in any dimension $d \ge 1$,
 assuming  that $\alpha \ge 1$ and
 $b$ is bounded and $\beta$-H\"older continuous
  with $\displaystyle { \beta > 1- \alpha/2}$. Our proof
 is different from
 the one in \cite{Tanaka} and is inspired by \cite{FGP}.
  The  assumptions on the $\alpha$-stable L\'evy process $L$
  which we consider are collected in Section 2 (see in particular
  Hypothesis \ref{nondeg}). Here we only
  mention
  two significant  examples which satisfy our hypotheses.
  The first  is  when $L = (L_t)$ is a
 standard $\alpha$-stable
 process (symmetric and rotationally invariant), i.e.,
  the characteristic function of the random variable $L_t$ is
  \beq \label{stand}
 E [e^{i \langle L_t , u\rangle }]
= e^{-t c_{\alpha} |u|^{\alpha}}, \;\;\; u \in \R^d, \; t \ge 0,
\eeq
 where $c_{\alpha}$ is a positive constant.
The second example is $L= (L^1_t , \ldots, L^d_t)$, where $L^1$,
 $\ldots, L^d$ are independent
 one-dimensional symmetric stable
 processes of index $\alpha$.
 In this case
 \beq \label{stand1}
 E [e^{i \langle L_t , u\rangle }]
  = e^{ - t k_{\alpha} (
  |u_1|^{\alpha} + \ldots + |u_d|^{\alpha})},
  \;\; u \in \R^d, \; t \ge 0,
\eeq
 where $k_{\alpha}$ is a positive constant.
 Martingale problems for
  SDEs driven by  $(L^1_t , \ldots, L^d_t)$
  have been recently studied
 (see \cite{basschen} and  references therein).

 We prove  the following result.

\begin{theorem}
\label{uno} Let $L$ be a symmetric $\alpha$-stable process
 with $\alpha \in [1, 2)$, satisfying Hypothesis \ref{nondeg}
 (see Section 2).
 Assume that
 $b \in C_{b}^{\beta}\left(  \mathbb{R} ^{d} ; \R^d \right) $
  for some $\displaystyle{ \beta \in (0,1)}$ such that
$$
 \beta > 1 - \frac{\alpha}{2}.
$$
Then pathwise uniqueness holds for equation (\ref{SDE}).
 Moreover, let $X^x = (X_t^x)$ be the solution starting at $x \in\mathbb{R}^{d}$. We have:

\hh (i) for any $t \ge 0,$ $p \ge 1$, there exists
  a constant $C(t,p)>0$ (depending also on $\alpha, \beta$ and $L=(L_t))$  such that
 \beq \label{ciao}
E[\, \sup_{0 \le s \le t} |X_s^x - X_s^y|^p\, ] \, \le C(t,p) \,\,
 |x-y|^p,\;\;\; x,\, y \in \R^d;
  \eeq
(ii) for any $t \ge 0$, the mapping: $x \mapsto X_t^x$ is a  homeomorphism from $\R^d$ onto $\R^d$, $P$-a.s.;

\hh(iii) for any $t \ge 0$, the mapping: $x \mapsto X_t^x$ is
 a $C^1$-function
 %differentiable
  on $\R^d$, $P$-a.s..
\end{theorem}
  All these assertions
 % All the assertions (i)-(iii)
   require   that
  $L$ is non-degenerate.
 Estimate \eqref{ciao} replaces the standard
 Lipschitz-estimate which holds
 without expectation $E$ when $b$ is Lipschitz continuous.
 % (in such case $C(t,p)$ is also independent of $L$).
 Assertion (ii) is the  so-called  homeomorphism
 property of  solutions (we refer to \cite{A}, \cite{protter} and \cite{Ku}; see also
    \cite{Za1} for the case of  Log-Lipschitz coefficients).
   Note that  existence of strong solutions
   for \eqref{SDE}
 follows easily
 by a compactness argument
   (see
 the comment before Lemma \ref{dis}).
  We mention  that   existence of
 weak solutions when $b$ is only measurable and bounded is
 investigated in \cite{Kurekov}.
 Since $C_{b}^{\beta'} (  \mathbb{R} ^{d} , \R^d )
  \subset C_{b}^{\beta} (  \mathbb{R} ^{d} , \R^d ) $
 when $0 < \beta \le \beta'$,
 our uniqueness result holds true for any
  $\alpha \ge 1$  when  $
  \beta  \in (1/2, 1).$
  Theorem \ref{uno} implies the existence
 of a stochastic flow for \eqref{SDE}
 (see Remark \ref{flow1})
 and  gives
  a partial   answer to a question posed  by L. Mytnik.

  The proof of the main result is given in Section 4. As in
  \cite{FGP} our method
 is based on
   an It\^o-Tanaka trick which requires suitable analytic regularity
   results.
 Such  results are proved in Section 3. They
  provide global
   Schauder estimates
     for the following resolvent equation on  $\R^d$
 \begin{equation} \label{resolv}
 \lambda u -
 {\cal L} u  - b \cdot Du = g,
\end{equation}
 where $\lambda>0$ and  $g \in C^{\beta}_b (\R^d)$
  are given and we assume
  $\alpha \ge 1$ and $\alpha + \beta >1$.
   Here
 $\cal L$ is the generator of the L\'evy
  process $L$  (see \eqref{gener}, \cite[Section 6.7]{A}
  and \cite[Section 31]{sato}).
 If $L $ satisfies  \eqref{stand} then ${\cal L}$
 coincides with the
 fractional Laplacian
 $-(- \triangle)^{\alpha/2}$
  on infinitely differentiable functions $f $
 with compact support   (see \cite[Example 32.7]{sato}),
  i.e., for any $x \in \R^d,$
 \beq \label{sch}
-(- \triangle)^{\alpha/2} f(x) = \int_{\R^d} \big(
 f (x+y) -  f (x)
   - 1_{ \{  |y| \le 1 \} } \, \,  y \cdot D f (x)   \, \big)
   \frac{\tilde c_{\alpha}}{|y|^{d+ \alpha}}dy.
    \eeq
 Schauder estimates are simpler to  prove when
  $\alpha >1$.
  In such a case, assuming  in addition  that
  ${\cal L} = -(- \triangle )^{\alpha/2}$, i.e.,
  $L $
   is a standard
   $\alpha $-stable process,
   %(see \eqref{stand}),
  these
    estimates can be  deduced  from the theory of
    fractional powers of
       sectorial operators (see \cite{L}). We also mention
       \cite[Section 7.3]{bassJFA} where  Schauder
    estimates  are proved when  $\alpha >1$
        and  ${\cal L}$ has the form \eqref{sch}
      but with variable coefficients, i.e.,
      $\tilde c_{\alpha } = \tilde c_{\alpha} (x,y)$.
   The limit case $\alpha =1$ in \eqref{resolv}
    requires a
   special attention even  for the fractional Laplacian
   ${\cal  L} = -(- \triangle )^{1/2}$.
    Indeed in this case ${\cal L}$ is of the ``same order'' of
     $b \, \cdot  D$. To treat $\alpha =1$, we use a localization
     procedure which is based on Theorem \ref{sta} where Schauder
     estimates are proved in the case of  $b(x) = k$, for
     any $x \in \R^d$,
  showing that the Schauder constant is independent of $k$
   (see also Remark \ref{al}).

  In order to prove Theorem \ref{uno}, in Section 4
   we   apply    It\^o's formula to  $u(X_t)$, where
   $u \in C^{\alpha + \beta}_b$ comes from Schauder estimates
    for \eqref{resolv} when $g =b$ (in such case \eqref{resolv}
    must be understood  componentwise).
  This is needed to perform the It\^o-Tanaka trick
 and find a new equation for $X_t$ in which
 the singular
  term $\int_0^t b(X_s)ds$  of \eqref{SDE} is replaced
 by more regular terms. Then uniqueness and
  \eqref{ciao} follow by
    $L^p$-estimates
   for stochastic integrals.
   Such estimates  require the deterministic
 Lemma \ref{dis} and the condition
$\displaystyle{ \alpha/2 + \beta
>1}$. In addition, properties (ii) and (iii)
 are obtained  transforming \eqref{SDE}
 into a form suitable for applying the results in \cite{Ku}.
 % We stress  that our uniqueness result  holds
 %  more generally if
 %  in \eqref{SDE} we consider  also  terms which satisfy a
 % Lipschitz-type condition
 % (see  Remark \ref{lip}).

 We will use the letter $c$ or $C$ with subscripts for
 finite positive
 constants whose  precise  value is unimportant;
  the constants   may
 change from
    proposition to proposition.

\section {Preliminaries and notation}

 General references for this section are \cite{A},
  \cite[Chapter 2]{SamTaqqu}, \cite{sato} and \cite{Zab}.

Let  $\lan u, v \ran $ (or $u\cdot v$)  be  the euclidean inner
product between $u$ and $v \in \R^d$, for any $d \ge 1$;
  moreover
$|u| = \lan u,u\ran^{1/2}$. If  $D \subset \R^d$
 we denote by $1_D$ the indicator function of $D$. The
  Borel $\sigma$-algebra of $\R^d$ will be indicated by
   ${\cal B}(\R^d)$. All the measures  considered in the sequel will be positive and Borel.
    A
    measure $\gamma$ on $\R^d$ is called symmetric if
   $\gamma (D) = \gamma (-D)$, $D \in {\cal B}(\R^d)$.

 \vskip 1mm Let us fix $\alpha \in (0,2)$.
 In \eqref{SDE} we
 consider
  a $d$-dimensional \textit{ symmetric $\alpha $-stable  process}
     $L=(L_t)$,  $d \ge 1$,
      defined on a fixed
  stochastic basis
  $(\Omega, {\cal F}, ({\cal F}_t)_{t \ge 0}, P)$ and
  ${\cal F}_t$-adapted; the
  stochastic basis  satisfies
  the usual assumptions (see
   \cite[page 72]{A}).
   Recall that   $L$ is a
   L\'evy process
    (i.e.,
  it is  continuous in probability,
  it  has  stationary increments,
     c\`adl\`ag  trajectories, $L_t - L_s$ is independent of
     ${\cal F}_s$, $0 \le s \le t$,
      and $L_0=0$) with the additional
     property that its  characteristic function is given by
 \beq \label{itol}
 E[ e^{i \langle  L_t ,  u\rangle }] =
 e^{- t \psi(u)}, \;\;  \;
  \psi(u)= - \int_{\R^d} \Big(  e^{i \langle u,y \rangle }
   - 1 - \, { i \langle u,y
\rangle} \, 1_{ \{ |y| \le 1 \} } \, (y) \Big ) \nu (dy),
  \eeq
 $u \in \R^d$, $t \ge 0$,  where
   $\nu$
   is a    measure such that
 \beq \label{spec}
 \nu (D) = \int_{\S} \mu(d \xi) \int_0^{\infty} 1_D (r \xi)
 \frac{dr}{r^{1+ \alpha}}, \;\;\; D \in {\cal B}(\R^d),
 \eeq
 for some symmetric, non-zero  finite measure $\mu$
  concentrated on the
  unitary sphere $\S =  \{ y \in \R^d \, : \, |y| = 1 \} $
  (see \cite[Theorem 14.3]{sato}).
  %  it follows that
  % $\nu$ is symmetric as well.

  {\vskip 1mm }The measure $\nu$ is called the
 L\'evy (intensity) measure
 of $L$ and \eqref{itol} is a special case of  the
  L\'evy-Khintchine  formula.
  The measure $\nu$  is
 a $\sigma$-finite  measure
on $\R^d$ such that $ \nu (\{ 0\})=0$ and
 $  \int_{\R^d} (1 \wedge
|y|^2 ) \, \nu(dy)
 <\infty,$
 with  $1 \wedge |\cdot | = \min (1, |\cdot|)$.
   Note that formula
  \eqref{spec} implies that
   \eqref{itol} can be rewritten as
$$
  \psi(u)= - \int_{\R^d} \big( \cos (\langle u,y \rangle)
   - 1 \big ) \nu (dy)
$$
\beq \label{bo} =
  - \int_{\S} \mu (d \xi)
  \int_0^{\infty} \frac{\cos (\langle u, r \xi \rangle)
   - 1 }{r^{1+ \alpha}} \; dr =  c_{\alpha}
    \int_{\S} |\langle u,  \xi \rangle |^{\alpha} \mu (d \xi),
    \;\; u \in \R^d
 \eeq
(see also \cite[Theorem 14.13]{sato}).
 The  measure $\mu$ is
 called the spectral measure
 of the stable process $L$. In this paper we make the
  following  non-degeneracy  assumption
 (cf. \cite{sztonik} and  \cite[Definition 24.16]{sato}).

\begin{hypothesis} \label{nondeg}
{\em  The support of the spectral measure
  $\mu$  is not contained in a proper linear
  subspace
  of $\R^d$.}
\end{hypothesis}
 It is not difficult to show that Hypothesis \ref{nondeg}
 is equivalent to the following assertion: there exists a positive
 constant
  $C_{\alpha}$ such that, for any $u \in \R^d$,
\beq \label{stim}
 \psi (u) \ge C_{\alpha} |u|^{\alpha}.
\eeq
 Condition \eqref{stim} is
 also assumed in \cite[Proposition 2.1]{kolokoltsov}.
 To see that \eqref{stim} implies
 Hypothesis \ref{nondeg}, we argue by contradiction:
  if Supp$(\mu) \subset (M
 \cap \S)$ where
  $M $ is the hyperplane containing all
    vectors orthogonal to some $u_0
   \not =0$, then $\psi (u_0)=0$. To show the converse,
   note that Hypothesis \ref{nondeg} implies that
    for any $v \in \R^d$ with $|v|=1$, we have $\psi (v)>0$
    (indeed, otherwise, we would have $\mu (\{ \xi \in \S \, :\,
     $ $ |\lan v, \xi\ran | >0\}) $ $=0$ which contradicts the
     hypothesis). By using a compactness argument,
      we deduce   that
     \eqref{stim} holds for any $u \in \R^d$ with $|u|=1$. Then,
     writing, for any $u \in \R^d$, $u \not =0$,
 $
 \int_{\S} |\langle u,  \xi \rangle |^{\alpha} \mu (d \xi)
 $ $= |u|^{\alpha} \int_{\S}
 \big|\langle \frac{u}{|u|},
  \xi \rangle \big|^{\alpha} \mu (d \xi)
$,
 we obtain easily \eqref{stim}.

 The infinitesimal generator $\cal L$ of the  process $L$ is given by
 \beq \label{gener}
 {\cal L}f(x) =
\int_{\R^d} \big(  f (x+y) -  f (x)
   - 1_{ \{  |y| \le 1 \} } \, \lan y , D f (x) \ran \big)
   \, \nu (dy), \;\; f \in C^{\infty}_c(\R^d),
 \eeq
 where $C^{\infty}_c(\R^d)$ is the space of all
  infinitely differentiable functions with compact support
 (see \cite[Section 31]{sato}). Let us come back on the
  examples  of $\alpha$-stable processes considered in
  Introduction
   which  satisfy
 Hypothesis \ref{nondeg}.
 The first  is  when $L$ is a
 standard $\alpha$-stable
 process,
 i.e.,
 $\psi (u) =  c_{\alpha} |u|^{\alpha}$.
  In this
  case $\nu $ has  density $\frac{C_{\alpha}}{|x|^{d+ \alpha}}$
  with respect to the Lebesgue measure in $\R^d$. Moreover
 the spectral measure $\mu$ is the normalized surface measure on
 $\S$ (i.e., $\mu$  gives
  a uniform distribution on $\S$;
  see \cite[Section 2.5]{SamTaqqu} and \cite[Theorem 14.14]{sato}).

 The second example is  $L= (L^1_t , \ldots, L^d_t)$,
 see \eqref{stand1}.
 In this case $\psi (u) = k_{\alpha} (|u_1|^{\alpha} + \ldots
+ |u_d|^{\alpha})$ and the L\'evy measure $\nu$ is more singular
since it is
 concentrated on the union  of the coordinates axes,
i.e., $\nu$ has density
$$
 c_1 \Big (1_{ \{ x_2=0, \ldots, x_d=0
 \}  } \frac{1}{|x_1|^{1+ \alpha}} + \ldots +
  1_{ \{ x_1=0, \ldots, x_{d-1}=0
 \}  } \frac{1}{|x_d|^{1+ \alpha}}
 \Big)
$$
 with respect to the Lebesgue measure.
 %(so it
 %is more singular than the
%previous case).
 The spectral measure $\mu$ is a linear combination
 of
 Dirac measures, i.e. $\mu = \sum_{k=1}^d (\delta_{e_k} +
 \delta_{-e_k})$, where $(e_k)$ is the canonical basis in $\R^d$.
 The generator is
$$
{\cal L}f(x) =  \sum_{k=1}^d \int_{\R} [ f (x+ s e_k) - f (x)
   - 1_{ \{  |s| \le 1 \} } \, s  \, \partial_{x_k}f (x)  ]
   \; \frac{c_{\alpha}}{|s|^{1+ \alpha}} ds,
   \; f \in C^{\infty}_c(\R^d).
$$
Let us fix some notation on  function spaces.

 We define  $C_{b}(\mathbb{R}
 ^{d};\mathbb{R}^{k})$, $k,\,d\geq1$, as set of all functions
 $f:\mathbb{R}^{d}\rightarrow\mathbb{R}^{k}$ which are bounded and
 continuous. It is a Banach space endowed with the
 supremum norm $\| f\|_0 = $ $\sup_{x \in \R^d}|f(x)|,$ $f \in
  C_{b}(\mathbb{R}
 ^{d};\mathbb{R}^{k}).$
 Moreover, $C_{b}^{\beta}(\mathbb{R}
 ^{d};\mathbb{R}^{k})$, $\beta \in (0,1)$,
   is the subspace of all $\beta$-H\"older continuous
   functions $f$, i.e., $f$ verifies
\begin{equation} \label{hol}
 [ f]_{\beta}:=\sup_{x\neq y\in\mathbb{R}^{d}}
 \frac{|f(x)-f(y)|}{|x-y|^{\beta}}<\infty.
\end{equation}
 $C_{b}^{\beta}(\mathbb{R}
 ^{d};\mathbb{R}^{k})$ is a Banach space with the norm
 $
  \| \cdot \|_{\beta} = \| \cdot \|_0 + [\cdot ]_{\beta}.
 $
 When $\R^k =\R$, we set
 $C_{b}^{{\beta}}({\mathbb{R}}^{d};{\mathbb{R}^k})
 = C_{b}^{{\beta}}({\mathbb{R}}^{d})$.
  Let $C_{b}^{0}(\mathbb{R}
 ^{d}, {\mathbb{R}}^{k} ) = C_{b}(\mathbb{R}
 ^{d}, {\mathbb{R}}^{k})$ and $[ \cdot ]_0 = \| \cdot\|_0$.
 For any $n \ge 1,$ $\alpha \in [0,1)$, we say that
 $f\in C_{b}^{n+{\alpha}}({\mathbb{R}}^{d})$
 if $f \in  C_{}^{n+{\alpha}}({\mathbb{R}}^{d}) \cap
  C_{b}^{\alpha}
 ({\mathbb{R}}^{d})$
 and, for
all $j=1,\dots,n$, the (Fr\'echet) derivatives $D^{j}f$
 $\in C_{b}^{{\alpha}}  ({\mathbb{R}}^{d};{(\mathbb{R}^{d})^{\otimes (j+1)}} )$.
 The space $C_{b}^{n+{\alpha}}({\mathbb{R}}^{d})$
 is a Banach space endowed with the
 norm
  $
\Vert f\Vert_{n+\alpha}$ $=\Vert f\Vert_{0}+\sum_{k=1}^{n}\Vert
D^{k}f\Vert _{0}+[D^{n}f]_{\alpha}$, $f\in
C_{b}^{n+{\alpha}}({\mathbb{R}}^{d})$.

\bre {\em Hypothesis \ref{nondeg} (or condition \eqref{stim})
 is
 equivalent to
 the following Picard's type condition (see \cite{Pi}):
 there exists $\alpha \in (0,2)$ and   $\displaystyle{
 C_{\alpha}>0}$,
   such that
  the following estimate
 holds, for any $\rho >0$, $u \in \R^d$ with $|u|=1$,
$$
 \int_{ \{ |\lan u, y \ran|\le \rho \} }
 |\lan u, y\ran|^2   \nu (dy) \ge
 C_{\alpha} {\rho}^{2 - \alpha }.
$$
The equivalence follows from the  computation
$$ \int_{ \{ |\lan
u, y \ran|\le \rho  \} } |\lan u, y\ran|^2
 \nu (dy)
= \int_{\S} |\lan u, \xi \ran |^2 \mu(d \xi) \int_0^{\infty} 1_{ \{
| \lan u, \xi \ran|\, \le \, \frac{\rho}{r} \} } \,  r^{1- \alpha}
dr
$$
$$
= \rho^{2- \alpha} \,
 \int_{\S} |\lan u, \xi \ran |^2 \mu(d \xi)
\int_{|\lan u, \xi \ran |}^{\infty} \frac{ds} {s^{3- \alpha}} =
 \frac{\rho^{2- \alpha}}{2- \alpha} \,
 \int_{\S} |\lan u, \xi \ran |^{\alpha} \mu(d \xi).
$$
The Picard's condition is usually imposed on the L\'evy measure
$\nu$ of a non-necessarily stable L\'evy process $L$ in order to
ensure that the
 law
of $L_t$, for any $t>0,$ has a $C^{\infty}$-density with respect to
the Lebesgue measure.
}\ere

\section {Some analytic regularity results  }

 In this section  we prove  existence of regular solutions
  to \eqref{resolv}. This result will be
   achieved through  Schauder estimates
   and will be important  in Section
  4 to prove uniqueness for \eqref{SDE}.

  We will use
  the following
   three properties of the $\alpha$-stable process $L$
 (in the sequel $\mu_t$ denotes
 the law of   $L_t$, $t \ge 0).$

\hh (a) $\mu_t (A) =  \mu_1 (t^{-1/\alpha} A)$, for any $A \in {\cal
B }(\R^d)$, $t>0$ (this scaling property follows  from
\eqref{itol}
 and \eqref{bo});

\hh {(b)} $\mu_t$ has a density $p_t $ with respect to the Lebesgue
measure, $t>0$; moreover $p_t \in C^1(\R^d)$ and its spatial
 derivative $Dp_t \in L^1(\R^d, \R^d)$
  (this is a consequence of Hypothesis
\ref{nondeg});

\hh {(c)} for any $\sigma > \alpha$, we have by \eqref{spec}
 \beq \label{tr}
 \int_{\{ |x| \le 1  \}} |x|^{\sigma} \nu (dx) < \infty.
\eeq
 The fact that (b) holds can be deduced by  an argument of
 \cite[Section 3]{sztonik}. Actually, Hypothesis \ref{nondeg}
 implies the following stronger result.

 \ble For any $\alpha \in (0,2)$, $t>0$, the density $p_t
  \in C^{\infty}(\R^d)$ and all
    derivatives $D^k p_t $ are integrable on $\R^d$,
   $k \ge 1$. \ele
\begin{proof}
We only show that $p_t \in C^{\infty}(\R^d)$ and  $Dp_t \in
L^1(\R^d, \R^d)$, following \cite{sztonik};
  arguing in a similar
 way one can obtain the full  assertion. By \eqref{stim},
 we know that $e^{-t \psi (u)} \le e^{ct|u|^{\alpha}}$,
   $u \in \R^d$, and so by the inversion formula of Fourier
 transform (see \cite[Proposition 2.5]{sato}) $\mu_t$ has a density $p_t \in L^1(\R^d) \cap C_0(\R^d)$,
 \beq \label{de}
 p_t (x) = \frac{1}{(2 \pi)^d} \int_{\R^d}
 e^{- i \lan x, z\ran} e^{-t \psi (z)} dz, \;\; x \in \R^d, \; t > 0.
\eeq
 Note that (a) implies that  $p_t (x) = t^{-d/\alpha} p_1 (t^{- 1/\alpha} x)$.
Thanks to \eqref{stim} one can differentiate infinitely many times
under the integral sign and see that $p_t \in C^{\infty} (\R^d)$.
 Let us fix $j=1, \ldots, d$ and check that
 the partial derivative $\partial_{x_j}
 p_t \in L^1(\R^d)$.
 By the scaling property (a) it is enough to consider $t=1$.
 By writing  $\psi = \psi_1 + \psi_2$,
$$
 \psi_1 (u) = - \int_{ \{ |y| \le 1\} } \big( \cos (\langle u,y \rangle)
   - 1 \big ) \nu (dy),\;\; \psi_2 =
   \psi - \psi_1,
$$
$$
 \partial_{x_j}p_1 (x) =  \frac{1}{(2 \pi)^d} \int_{\R^d}
 e^{- i \lan x, z\ran} \big( (-i z_j ) e^{- \psi_1 (z)}
 \big)  e^{- \psi_2 (z)}  dz, \;\; x \in \R^d.
$$
We find easily that  $\psi_1 \in C^{\infty}(\R^d)$
 and so, using also \eqref{stim}
 we deduce that $\psi_1$ is in the Schwartz space ${\cal S} (\R^d)$.
  In particular, there exists
 $f_1 \in L^1(\R^d) $ such that the Fourier transform
$\hat f_1 (z) = (-i z_j) e^{- \psi_1 (z)} $. On the other hand (see
\cite[Section 8]{sato}), there exists an infinitely divisible
probability measure $\gamma$ on $\R^d$ such that the Fourier
transform $\hat \gamma (z) = e^{- \psi_2 (z)}$. By \cite[Proposition
2.5]{sato} we infer that
 $\hat {f_1 * \gamma}$ $= \hat f_1 \, \cdot \hat \gamma$.
 By the inversion formula
 we deduce  that $\partial_{x_j}p_1 (x) =
 (f_1 * \gamma) (x)$
and this proves that  $\partial_{x_j}p_1 \in L^1(\R^d)$.
\end{proof}

Remark that (c) implies that the expression of  ${\cal L} f$
 in \eqref{gener} is  meaningful for any
 $f \in C^{1+ \gamma}_b (\R^d)$ with
 $1+ \gamma > \alpha$.  Indeed ${\cal L}f(x)$
 can be decomposed into the sum of two integrals, over
 $\{ |y| >1\}$ and $\{ |y| \le 1\}$ respectively.
 The  first integral is finite since  $f$ is bounded.
 To treat  the
 second one, we can use the estimate
 \begin{align} \label{rit}
 & | f(y + x) - f(x)
   -  \, y \cdot D f (x) |
\\ \nonumber
& \le \int_0^1 |D f (x + ry) -  D f (x)
 |\, |y| dr \le  \| Df\|_{\gamma} \, |y|^{1+  \gamma},
  \;\; |y| \le 1.
\end{align}
  Note that
 $ {\cal L} f
 \in C_b (\R^d)$ if $f \in  C^{1+ \gamma}_b (\R^d)$ and
 $1+ \gamma > \alpha$.

 \vskip 1mm The next result  is a  maximum principle.
  A related  result is in \cite[Section 4.5]{jacob}.
  This will be used  to prove
  uniqueness of solutions to \eqref{resolv} as well as
  to study  existence.

\begin{proposition} \label{max}   Let $\alpha \in (0,2)$. If
  $u \in C^{1+ \gamma}_b (\R^d)$, $1+ \gamma > \alpha $,
  is a solution
  to $\lambda u - {\cal L} u$ $- b \cdot Du = g$, with $\lambda >0$
   and $g \in C_b(\R^d)$, then
 \begin{equation} \label{max1}
 \| u \|_0 \le \frac{1}{ \lambda} \| g
 \|_0,\;\;\; \lambda>0.
\end{equation}
\end{proposition}
 \begin{proof}
 Since $-u$ solves the same equation of $u $
  with $g$ replaced by $-g$, it is enough to prove
 that $u(x) \le \frac{\| g\|_0}{\lambda}$, $x \in \R^d$.
 Moreover, possibly replacing  $u$ by
 $u - \inf_{x \in \R^d} u(x)$,
  we may assume that $u \ge 0 $.

 Now we  show that there exists $c>0$
 such that, for any $\epsilon >0$ we can find
 $u_{\epsilon} \in C^{1+ \gamma}_b(\R^d)$ with
 $\| u_{\epsilon} \|_0 =$ $ \max_{x \in \R^d}
 |u_{\epsilon} (x)|$ and also
$$
 \| u - u_{\epsilon} \|_{{1+ \gamma}} < \epsilon \, c.
$$
To this purpose let $x_{\epsilon} \in \R^d$ be such that
 $u(x_{\epsilon}) > \|u \|_0 - \epsilon$ and
  take a test function $\phi \in C^{\infty }_c (\R^d)$
   such that $\phi (x_{\epsilon}) =1$,
    $0 \le \phi \le 1$, and $\phi(x) =0$ if $|x-x_{\epsilon}| \ge
    1$. One checks that
 $
u_{\epsilon}(x) =  u(x) + 2 \epsilon \, \phi(x)
 $
verifies the assumptions.  Let us define the operator
  ${\cal L }_1  =  {\cal L}
 +  b \cdot D $
and write
 $$
\lambda u_{\epsilon}(x) -   {\cal L }_1 u_{\epsilon}(x)
 = g(x)
+ \lambda ( u_{\epsilon}(x) - u(x)) -
 {\cal L}_1 (u_{\epsilon} - u)(x).
$$  Let  $y_{\epsilon}$ be one point in which
 $u_{\epsilon}$ attains its global maximum.  Since clearly
 ${\cal L}_1 u_{\epsilon}(y_{\epsilon}) \le 0$, we have
 (using also \eqref{rit})
$$
 \lambda \| u_{\epsilon} \|_0  =
  \lambda u_{\epsilon}(y_{\epsilon}) \le
  \| g\|_0 + C \| u - u_{\epsilon} \|_{{1+ \gamma}}
   \le \| g\|_0 + C_1 \, \epsilon.
$$
Letting $\epsilon \to 0^+$, we get  \eqref{max1}.
 \end{proof}

 Next we prove  Schauder estimates for
  \eqref{resolv} when {\it $b$ is constant.} The case of $b
  \in C^{\beta}_b(\R^d, \R^d)$ will be treated in Theorem
  \ref{reg}.
  We stress  that the constant $c$ in \eqref{schaud}
   is independent of $b=k$.
   %This will be important
  %in the proof of Theorem \ref{reg} to deal with
  % the critical case  $\alpha =1$.

The  condition $\alpha + \beta >1$ which we impose
 is needed to  have a regular $C^1$-solution $u$.
  On the other hand,
  the next result holds more generally without the
  hypothesis $\alpha + \beta < 2 $;
  this is  imposed to simplify the proof
 and it is not restrictive in the study
  of
  pathwise  uniqueness for \eqref{SDE}.

\begin{theorem} \label{sta} Assume Hypothesis \ref{nondeg}.
 Let $\alpha \in (0,2)$ and
  $\beta  \in (0, 1)  $
 be such that $1 < \alpha + \beta < 2$.
   Then,  for any
 $\lambda>0$, $k \in \R^d$,   $g \in C^{\beta}_b (\R^d)$,
   there exists a unique   solution
  $u= u_{\lambda} \in C^{\alpha + \beta}_b(\R^d)$ to the equation
   \beq \label{we}
 \lambda u -  {\cal L} u -
  k \cdot Du   = g
 \eeq
 on $\R^d$ (${\cal L}$ is defined in \eqref{gener}$)$.
   In addition there exists a constant $c$ independent of
  $g$, $u$,  $k $ and $\lambda >0$ such that
\begin{equation} \label{schaud}
\lambda \| u\|_0 \, +
 \, \lambda^{\frac{\alpha + \beta - 1}{\alpha }}
 \| Du \|_{0} \, + \, [Du]_{\alpha + \beta
-1}\le c \| g \|_{{\beta}}.
\end{equation}
%Moreover we have
%\begin{equation}  \label{gh1}
%\| Du \|_{0 } \le C(\lambda)  \| g \|_{0},
%\;\; \text{ \rm with $C(\lambda) \to 0 $ as $\lambda \to + \infty$.}
%\end{equation}
\end{theorem}
 \begin{proof}
 Equation \eqref{we}
  is meaningful for $u \in C^{\alpha + \beta}_b (\R^d)$
with $\alpha + \beta>1$ thanks to \eqref{rit}.
 Moreover,  uniqueness follows from Proposition \ref{max}.

 To prove the result,
 we use the semigroup approach as in \cite{DL}.
  To this purpose, we  introduce  the
$\alpha$-stable Markov semigroup $(P_t)$ acting on $C_b(\R^d)$ and
associated to ${\cal L} + k \cdot Du$, i.e.,
$$
P_t f(x) = \int_{\R^d} f(z + tk)\,  p_t (z-x) dz, \;\; t>0, \; f \in C_b(\R^d),
 \; x \in \R^d,
$$
  where $p_t$ is defined in \eqref{de}, and
  $P_0 = I$.  Then we consider
 the bounded function $u = u_{\lambda}$,
 \begin{equation} \label{d}
 u(x) = \int_0^{\infty} e^{-\lambda t} P_t g(x) dt, \;\; x \in \R^d.
 \end{equation}
  We are going to show   that $u$ belongs to
    $C^{\alpha + \beta}_b (\R^d)$,  verifies
\eqref{schaud} and solves \eqref{we}.

\hh {\it I Part.} We prove that
$u \in C^{\alpha + \beta}_b (\R^d)$
and that \eqref{schaud} holds.

 First note that $\lambda \| u\|_0 \le \| g\|_0$
  since $(P_t)$ is a contraction semigroup.
 Then,
  using the scaling property
  $p_t (x) = t^{-d/\alpha} p_1 (t^{- 1/\alpha} x)$,
 we arrive at
\beq \label{se}
|DP_t f(x)| \le \frac{t^{-1/\alpha}}{t^{d/\alpha}}
\int_{\R^d} |f(z + t k)| \, | D p_1 (t^{-1/\alpha}z
 - t^{-1/\alpha}x )| \, dz \le \frac{c_0 \| f\|_0} {t^{1/\alpha}},
\eeq $t>0$, $f \in C_b (\R^d)$, where $c_0 = \| Dp_1
\|_{L^1(\R^d)},$ and so we find the  estimate
\begin{equation} \label{lar}
 \| DP_t f \|_0 \le \frac{c_0}{t^{1/ \alpha}} \| f \|_0,
  \;\; f \in C_b(\R^d),\; t>0.
 \end{equation}
  By interpolation theory we know that $
 \big ( C^{}_b (\R^d), C^{1}_b (\R^d) \big)_{\beta, \infty}
 = C^{\beta}_b (\R^d), $ $\beta \in (0,1)$,
 see for instance \cite[Chapter 1]{L};
   interpolating
 the previous estimate with the   estimate
 $\| DP_t f \|_0 \le  \| D f \|_0$, $t \ge 0,$
  $f \in C^1_b(\R^d)$, we  obtain
 \begin{equation} \label{est1}
 \| DP_t f \|_0 \le \frac{c_1}{t^{
  (1- \beta) / \alpha}}
\| f \|_{{\beta}}, \;\; t>0,\;\; f \in C^{\beta}_b(\R^d),
 \end{equation}
 with $c_1= c_1(\alpha, \beta)$. In a similar way, we  also find
\begin{equation} \label{est2}
 \| D^2 P_t f \|_0 \le \frac{c_2}{t^{
  (2- \beta) / \alpha}} \| f \|_{{\beta}}, \;\; t>0,\;\; f \in C^{\beta}_b(\R^d).
 \end{equation}
  Using \eqref{est1} and  the fact that
   $\frac{1- \beta}{\alpha} <1$, we can differentiate under the
   integral sign in \eqref{d} and
  prove
 that there exists
  $D u(x) = D u_{\lambda}(x)$, $x \in \R^d$.
  Moreover $D u_{\lambda}$
  is bounded on $\R^d$ and we have, for any $\lambda>0$ with
  $\tilde c$ independent of $\lambda$, $u$, $k$ and $g$,
  $$
\lambda^{\frac{\alpha + \beta - 1}{\alpha }}
 \| Du \|_{0} \le
 \tilde c \| g \|_{{\beta}}
  $$
   (we have used that
 $\int_0^{\infty} e^{-\lambda t} t^{-\sigma} dt $  $=
\frac{c}{\lambda^{1 -\sigma}}$, for $\sigma < 1$ and $\lambda>0$).

  It remains to prove that $Du \in C^{\theta}_b(\R^d, \R^d)$,
  where
  $\theta = \alpha - 1 + \beta  \in (0,1)$. We
   proceed as in the proof
   of \cite[Proposition 4.2]{bassJFA}
    and
  \cite[Theorem 4.2]{P}.

 Using \eqref{est1},  \eqref{est2}
    and the fact that $2 - \beta > \alpha$,
 we find, for any $x, x' \in \R^d$, $x \not = x',$
\begin{align*}
 |Du (x) - Du(x')| &\le C  \| g\|_{{\beta}}
   \Big ( \int_0^{|x-x'|^{\alpha}}
  \frac{1}{t^{
  (1- \beta) / \alpha}} dt +
   \int_{|x-x'|^{\alpha}}^{\infty}
 \frac{|x-x'|}{t^{
  (2- \beta) / \alpha}} dt
 \Big)
\\
& \le c_3  \| g\|_{{\beta}}
 {|x-x'|^{\theta}},
\end{align*}
 and so $[Du]_{ \alpha - 1 + \beta }
 \le c_3  \| g\|_{{\beta}}$,
  where $c_3 $ is independent of  $g$, $u$,
   $k $ and $\lambda$.

\vskip 1mm  \noindent  {\it II Part.}  We prove that $u$ solves
  \eqref{we}, for any $\lambda >0$.

 We use the fact that the semigroup $(P_t)$ is strongly continuous
on the Banach space $C_0 (\R^d) \subset C_b (\R^d)$ of all functions
vanishing at infinity (endowed with $\| \cdot \|_0$; see
 \cite[Section 6.7]{A} and \cite[Section 31]{sato}).
  Let
${\cal A}: D({\cal A}) \subset C_0 (\R^d) \to C_0 (\R^d)$ be its
generator. By \cite[Theorem 31.5]{sato})  $C^2_0 (\R^d) \subset
D({\cal A})$ and moreover  ${\cal A}f = {\cal L}f + k \cdot Df$ if
$f \in C^2_0 (\R^d)$ (we say that $f$ belongs to $C^2_0 (\R^d)
 \;$  if $f
\in  C^2_b (\R^d) \cap C_0(\R^d)$ and all its first and second
partial derivatives belong to $C_0 (\R^d)$).

 \noindent {\sl We first show  the assertion assuming in addition
 that $g \in C^2_0 (\R^d)$.}

  It is   easy to check that
 $u$  belongs to
 $C^2_0
(\R^d)$ as well.
 To this purpose, one can use the estimates $\| D^k
 P_t g \|_0 \le \| D^k g \|_0, $ $t \ge 0$, $k = 1, 2$, and the
dominated convergence theorem.
 By
  the Hille-Yosida theorem we know that $u \in D({\cal A})$ and $
 \lambda u - {\cal A} u = g $. Thus we have found that
  $u$ solves \eqref{we}.

 \hh {\sl Let us  prove the assertion when
    $g \in C^2_b (\R^d)$.}

  Note  that also $u \in C^2_b (\R^d)$.
  We  consider   a function $\psi \in
C_c^{\infty}(\R^d)$ such that
 $\psi (0)=1$ and
  introduce $g_n (x) = \psi(x/n) g(x)$,
  $x \in \R^d$, $n \ge 1$. It is clear that $g_n, \,
  u_n \in C^2_0(\R^d)$ ($u_n$ is given in \eqref{d} when $g$
  is replaced by $g_n$).
 We know that
 \beq \label{sol}
  \lambda u_n(x)- {\cal L}u_n(x) - k \cdot Du_n(x) = g_n(x),
   \;\; x \in \R^d.
\eeq
 It is easy to see that
  there exists $C>0$ such that $\| g_n\|_{2} \le C$, $n \ge
 1$, and moreover
  $g_n $ and $Dg_n$ converge pointwise  to $g$ and $Dg$
  respectively. It follows that also
  $\| u_n\|_{2} $ is uniformly bounded
  and moreover
  $u_n $ and $Du_n$ converge pointwise  to $u$ and $Du$
  respectively.
  Using also \eqref{rit},
  we can apply  the dominated convergence theorem and  deduce that
$$
 \lim_{n \to \infty} {\cal L}u_n(x) = {\cal L}u(x),\;\; x \in \R^d.
$$
Passing to the limit in \eqref{sol},
 we obtain that $u$
 is a solution to \eqref{we}.

\vskip 0.5mm \noindent {\sl Let now $g \in C^{\beta}_b (\R^d)$.}

  Take any $\phi \in C_c^{\infty}(\R^d)$ such that
  $0 \le \phi \le 1$ and $\int_{\R^d } \phi =1$. Define
  $\phi_n (x) = n^d\phi(xn)$ and $g_n = g * \phi_n$.
 Note that  $(g_n) \subset C^{\infty}_b(\R^d)
   = \cap_{k \ge 1} C_b^k(\R^d)$ and
 $\| g_n\|_{\beta} \le \| g\|_{\beta} $,
 $n \ge 1$. Moreover,
 possibly  passing to a subsequence
   still denoted by
   $(g_n)$,  we may assume that
  \beq \label{conv}
 g_n \to g  \;\; \text{in} \;\; C^{\beta'}(K).
 \eeq
 for any compact set $K \subset \R^d$ and $0 < \beta' < \beta$
 (see page 37 in \cite{Kr}).
 Let $u_n$ be  given in \eqref{d} when $g$
  is replaced by $g_n$. By the first part of the proof, we know
  that
 $$
 \|  u_n\|_{\alpha + \beta} \le C \| g_n \|_{\beta} \le
  C \| g\|_{\beta},
$$
 where $C$ is independent of $n$. It follows that, possibly
 passing to a subsequence  still denoted with $(u_n)$,
   we have that
  $
 u_n \to u $  in $  C^{\alpha + \beta'}(K),$
 for any compact set $K \subset \R^d$ and $\beta' >0$
  such that $1 < \alpha +
 \beta' < \alpha + \beta.$
  Arguing as before, we can
 pass to the limit in
  $ \lambda u_n(x)- {\cal L}u_n(x) - k
   \cdot Du_n(x)$ $ = g_n(x)$
 and obtain that $u$ solves \eqref{we}.
 The proof is complete.
\end{proof}

Now we extend Theorem \ref{sta}
 to the case in which $b$ is H\"older
 continuous. We can only do this
  when $\alpha \ge 1$ (see also
   Remark \ref{al}).
% Note that the case $\alpha >1$
% is  simpler than the critical case $\alpha =1$ (cf.
% also \cite[Section 7.3]{bassJFA}).
% Indeed, when $\alpha =1$, the operator  $\cal L $
%  (which is comparable with  $-(- \triangle)^{1/2})$
%  is ``of the same order as''
%    $ b\cdot D u $.
 To prove the result when $\alpha =1$ we  adapt the
localization procedure  which is well known
   for second order uniformly elliptic
operators with H\"older continuous coefficients (see
 \cite{Kr}). This technique works in our situation since
 in estimate \eqref{schaud} the constant is independent
 of $k \in
\R^d$.

 The next proof   requires
 the following  interpolatory
 inequalities
  (see \cite[page 40, (3.3.7)]{Kr});  for any $t \in
 [0,1)$,
 $0 \le s \le r <1 $, there exists $N = N(d,k, r,t)$
 such that if
 $f \in C_b ^{\, r+t} (\R^d, \R^k)$,
 then
 \beq \label{kry11}
 [f]_{s+t} \le N [f]_{r+t}^{s/r} \;\,  [f]_t^{1- \, s/r},
 \eeq
where $[f]_{s+t}$ is defined as in \eqref{hol}
if $0< s+t <1$, $[ f
]_0 = \| f \|_0$, $[f]_1 = \| Df\|_0$, and
 $[f]_{s+t} = [Df]_{s+t - 1}$ if $ 1 < s+t < 2$. By
 \eqref{kry11} we deduce, for any $\epsilon >0$,
\beq \label{kry22}
 [f]_{s+t} \le \tilde N  \epsilon^{r-s}
 [f]_{r+t} \; + \; \tilde N \epsilon^{-s} [f]_t,
 \;\;\; f \in C_b ^{\, r+t} (\R^d, \R^k).
 \eeq

\begin{theorem} \label{reg} Assume Hypothesis \ref{nondeg}.
  Let $\alpha \ge 1$  and
 $\beta  \in (0, 1)  $
 be such that $1<  \alpha + \beta < 2 $.
 Then,  for any
 $\lambda>0$, $g \in C^{\beta}_b (\R^d)$,
   there exists a unique  solution
  $u= u_{\lambda} \in C^{\alpha + \beta}_b
  (\R^d)$ to the equation
  \beq \label{wee}
 \lambda u -  {\cal L} u  - b \cdot Du = g
 \eeq
 on $\R^d$. Moreover, for any  $\omega >0$,
  there exists  $c = c(\omega) $, independent
  of  $g$ and  $u$,  such that
 \begin{equation} \label{sch4}
 \lambda \| u\|_0 + [Du]_{\alpha + \beta - 1}
  \le c \| g\|_{\beta}, \;\;  \lambda \ge \omega.
 \end{equation}
 Finally,  we have
 $\lim_{ \lambda \to \infty} \| Du_{\lambda} \|_{0} =0$.
\end{theorem}
 \begin{proof}
 Uniqueness and  estimate
 $\lambda \| u\|_0 \le \| g\|_0$, $\lambda >0,$
  follow from the maximum principle
 (see Proposition \ref{max}).
  Moreover, the last assertion
  follows from \eqref{sch4} using \eqref{kry11}.
   Indeed, with $t = 0$, $s =
  1 $, $r = \alpha + \beta $, we obtain,
  for $\lambda \ge \omega$,
 $$
[Du_{\lambda}]_{0} = [u_{\lambda}]_{1}   \le N
[Du_{\lambda}]_{\alpha + \beta -1}^{\frac{1}{\alpha + \beta}}
\;
  [u_{\lambda}]_{0}^{ 1 - \frac{1}{\alpha + \beta}}
 \le N \tilde c \,\, \lambda^{- {\frac{\alpha + \beta - 1
 }{\alpha +
\beta}}  } \; \| g\|_{\beta},
$$
where $\tilde c= \tilde c(\omega)$. Letting
 $\lambda \to \infty$, we get the assertion.

 Let us prove   existence and estimate $[Du]_{\alpha + \beta - 1}
  \le c \| g\|_{\beta},$ for $ \lambda \ge \omega$,
  with $\omega >0$
    fixed.
 We treat $\alpha >1$
  and $\alpha = 1$
 separately.

\vskip 1mm  \noindent  {\it I Part (the case $\alpha >1)$}.
 In the sequel we will use the  estimate
 \beq \label{pri}
 \|  l f \|_{\theta} \; \le \;
    \| l \|_{ 0}  \| f \|_{\theta }\, + \,
   \| f \|_{0 }  [ l ]_{\theta}, \;\;
       l, \,f \in { C}_{b}^{\theta}(\R^d),\;\;
        \theta \in (0,1).
\eeq Writing $ \lambda u(x) -
  {\cal L} u(x)
=   g(x) +  b(x)\cdot D u(x) $, and using \eqref{schaud}
 and \eqref{pri},
    we obtain
 the following  estimate (assuming
  that
$u \in C^{\alpha + \beta}_b (\R^d)$ is a solution  to
 \eqref{wee})
 \begin{align} \label{stima}
  [ D u ]_{\alpha + \beta -1} & \le C \| g \|_{\beta}  +
   C \| b \cdot D u \|_{\beta}
   \\ \nonumber & \le
    C \| g \|_{\beta}  + C \| b \|_{\beta}  \| Du \|_{0}
    + C \| b \|_{0}
    [ Du ]_{\beta},
 \end{align}
 where $C$ is independent of $\lambda>0$.
  Combining the interpolatory  estimates (see \eqref{kry22}
 with $t=0$, $s = 1 + \beta$, $r = \alpha + \beta$)
$$ [ Du  ]_{\beta}
 \le
 \tilde N \epsilon^{\alpha -1} [ Du ]_{{\alpha + \beta -1}}
 + \tilde N \epsilon^{- (1 +\beta) }\| u \|_{0},
 \;\; \epsilon >0,
 $$
 and $ \|  Du  \|_{0}
 \le
 \tilde N \epsilon^{\alpha + \beta -1}
 [ Du ]_{{\alpha + \beta -1}}
 + \tilde N \epsilon^{- 1  }\| u \|_{0}$
 (recall that $\alpha + \beta  > 1+ \beta$)
  with  the maximum principle,
 we get for $\epsilon$ small enough the a-priori estimate
\begin{align} \label{stima1}
 \lambda \| u\|_0  +
 [ Du ]_{\alpha + \beta -1} &\le c_1 (\| g \|_{\beta} +
  C(\epsilon) \| u\|_0)
  \\ \nonumber &\le c_1 \big(\| g \|_{\beta} +
  \frac{C(\epsilon)}{\omega} \| g\|_0 \big) \le
    C_1 \|g \|_{\beta},
 \end{align}
 for any  $\lambda \ge \omega$.
 Now to prove  the existence of
   a $C^{\alpha + \beta}_b$-solution,
  we  use the
   continuity method
 (see, for instance, \cite[Section 4.3]{Kr}).
 Let us introduce
 \begin{equation} \label{drr}
  \lambda u(x) -  {\cal L} u(x)  -
    \delta b(x)\cdot D u(x) = g(x),
 \end{equation}
$x \in \mathbb{R}^d,$ where $\delta \in [0,1]$ is a parameter. Let
 us define $ \Gamma = \{  \delta  \in [0,1] \, :\,  $  there is a
   unique solution $u =u_{\delta}
  \in C^{\alpha + \beta}_b(\mathbb{R}^d)$,
    for any $
   g \in C^{\beta}_b(\mathbb{R}^d)\}. $

Clearly $\Gamma$ is not empty since $0 \in \Gamma. $
 Fix $\delta_0 \in \Gamma$ and rewrite \eqref{drr}
  as
$$
\lambda u(x) -  {\cal L} u(x)  -
    \delta_0 b(x)\cdot D u(x) = g(x)
    + (\delta - \delta_0) b(x) \cdot Du(x).
$$
Introduce the operator $S : C^{\alpha + \beta}_b (\mathbb{R}^d)
 \to
C^{\alpha + \beta}_b(\mathbb{R}^d).
 $ For any $v \in  C^{\alpha + \beta}_b(\mathbb{R}^d)$,
  $u = S v$ is the unique $C^{\alpha + \beta}_b $-solution to
   $\lambda u(x) -  {\cal L} u(x)$ $  -
    \delta_0 b(x)\cdot D u(x) = g(x)$ $
    + (\delta - \delta_0) b(x) \cdot Dv(x).$

  By using   a-priori estimate \eqref{stima1},
 we find that $ \| Sv_1 - Sv_2 \|_{\alpha + \beta}
 \le $ $\displaystyle {2|\delta - \delta_0|}$
 $ \cdot \, \tilde c_1 \, \|b \|_{\beta}
 \| v_1 - v_2 \|_{\alpha + \beta}$. By choosing $|\delta
  - \delta_0|$ small enough, $S$ becomes a contraction and it has
  a unique fixed point which is the solution to \eqref{drr}.
   A compactness argument shows that $\Gamma = [0,1]$.
 The assertion is proved.

\vskip 1mm \noindent  {\it II Part (the case $\alpha =1)$}.
 As before, we establish
    the existence of a $C^{1+ \beta}_b (\R^d)$-solution,
   by using the
  continuity method. This requires
   an a-priori
  estimate \eqref{stima1} for $\alpha =1$.

 Let $u \in C^{1+ \beta}_b (\R^d)$ be a solution.
  Let $r>0$. Consider a function $\xi \in C_c^{\infty}(\R^d)$
 such that $\xi (x) = 1$ if $|x| \le r$ and
$\xi (x) = 0$ if $|x| > 2r$.

 Let now $x_0 \in \R^d$ and define $\rho (x) = \xi (x-x_0)$,
 $x \in \R^d$, and  $v = u \rho$.
 One can easily check that
 \begin{equation} \label{vai}
 {\cal L} v (x) = \rho(x)
 {\cal L} u(x) +
u(x) {\cal L} \rho (x)
\end{equation}
$$
+ \int_{\R^d} (\rho (x+ y) - \rho(x))
 (u(x+y) - u(x)) \, \nu(dy), \;\; x \in \R^d.
$$
We have
$$
\lambda v(x) -
  {\cal L} v(x) -
   b(x_0) \cdot Dv(x) =
f_1(x)+f_2(x)+f_3(x) + f_4(x), \;\; x \in \R^d,
$$ where
\[ f_1(x)=\rho(x)g(x), \qquad f_2(x)=
  (b(x)-
b(x_0)) \cdot Dv(x), \]
\[ f_3(x)=- u (x) [{\cal L} \rho (x) +
b(x) \cdot D \rho(x)],
 $$$$\;\; f_4 (x) = - \int_{\R^d} (\rho (x+ y) - \rho(x))
 (u(x+y) - u(x)) \, \nu(dy),\;\; x \in \R^d.
\]
 By Theorem \ref{sta} we know that
\begin{equation}\label{6.6}
 [ D v ]_{ \beta }
 \leq C_1(\|f_1\|_{\beta} + \|f_2\|_{\beta} +
\|f_3\|_{\beta} + \|f_4\|_{\beta}),
\end{equation}
 where the constant $C_1$ is independent of $x_0$ and
  $\lambda$.
 Let us consider the crucial  term $f_2$.
  By \eqref{pri}
 we find
$$
 \|  f_2 \|_{\beta } \; \le \;  \big( \sup_{x \in B(x_0, 2r)}
 |b(x) - b(x_0)| \big)\, [ Dv ]_{\beta} +
  \| Dv \|_0 \| b\|_{\beta}.
$$
Let us fix  $r$  small enough such that
 $C_1  \sup_{x \in B(x_0, 2r)}
 |b(x) - b(x_0)| < 1/2$.  We get
 \begin{equation} \label{d4}
 [ D v ]_{ \beta }
 \leq 2 C_1(\|f_1\|_{\beta} +  \| Dv \|_0 \| b\|_{\beta}  +
\|f_3\|_{\beta} + \|f_4\|_{\beta}).
 \end{equation}
 Note that $\| f_1\|_{\beta}   $
  $\le C(r) \, \| g\|_{\beta} $.
 Using  again the  interpolatory
  estimates  \eqref{kry22}
 together with the maximum principle, we
  arrive at
$$
 [ Dv ]_{ \beta }
 \leq  C_2 ( \|g\|_{\beta} + \| f_3\|_{\beta} +
  \| f_4\|_{\beta}),
$$
 for any $\lambda \ge \omega$.
 Let us estimate  $f_4$.
  To this
  purpose we introduce the following non-local
  linear operator $T$
 $$
 T f(x) = \int_{\R^d} (\rho (x+ y) - \rho(x))
 (f(x+y) - f(x)) \, \nu(dy), \;\, f \in C^1_b (\R^d),
  \;  x \in \R^d.
$$
 One can easily check  that
 $T$ is continuous from  $ C^1_b(\R^d)$
  into  $C_b(\R^d) $ and from $C^{1+ \beta}_b(\R^d)$
  into $  C_b^1(\R^d) $. To this purpose
  we only remark that, for any $x \in \R^d$,
$$
|DTf(x)| \le 5 \, \| \rho\|_{2} \| f\|_{1} \big(
 \int_{\{ |y| \le 1 \}}|y|^2 \nu(dy) +
  \int_{\{ |y|>1 \} } \nu(dy)\big)
$$
$$
+  5\, \| \rho\|_1 \| f\|_{1+ \beta} \big(
 \int_{\{|y| \le 1 \} }|y|^{1+ \beta} \nu(dy) +
  \int_{\{ |y|>1 \}}
  \nu(dy)\big), \;\; f \in C^{1+ \beta}_b (\R^d).
$$
 By interpolation theory we know that
$$
 \Big ( C^{1}_b (\R^d), C^{1+ \beta}_b (\R^d)
  \Big)_{\beta, \infty}
 = C^{1 + \beta^2}_b (\R^d),
$$
  see
 \cite[Chapter 1]{L}, and so  we get that
 $
 T$ is continuous from  $ C^{1+ \beta^2}_b(\R^d)
 $ into  $ C_b^{\beta}(\R^d)$ (see
 \cite[Theorem 1.1.6]{L}). Since $f_4 = - Tu$,
 we obtain
  the estimate
$$
 \| f_4\|_{\beta} \le C_3 \| u \|_{{1+ \beta^2}}.
$$
 We have $ \| f_4\|_{\beta} +
  \| f_3\|_{\beta} \le c_{3}(r) \, \| u \|_{{1+ \beta^2}}$
   and so
$$
 [ Dv ]_{ \beta }
 \leq  C_4 ( \|g\|_{\beta} +
 \| u \|_{{1+ \beta^2}}).
$$
 It follows that $ [ D u ]_{ C^{\beta}(B(x_0, r))
 }
 \leq  C_4 ( \|g\|_{\beta} +
  \| u \|_{{1+ \beta^2}})$,
  where $B(x_0, r)$ is the ball of center $x_0$
   and radius $r>0$.
 Since $C_4$ is
  independent of $x_0$,  we obtain
$$
 [D u ]_{ \beta }
 \leq  C_4 ( \|g\|_{\beta} +
 \| u \|_{{1+ \beta^2}}),
$$
 for any $\lambda \ge \omega$.
 Using again  \eqref{kry22} and the maximum pinciple,
  we get the
 a-priori estimate \eqref{stima1} for $\alpha =1.$
 Applying the continuity method we obtain the assertion.
 The proof is complete.
\end{proof}

\begin{remark} \label{al}
 {\em In contrast  with Theorem \ref{sta},
 in Theorem \ref{reg} we can not show existence of
 $C^{\alpha + \beta}_b$-solutions
  to \eqref{wee} when $\alpha <1$.
  The difficulty   is evident from
 the a-priori estimate \eqref{stima}.
  Indeed, starting from
  $$
 [ D u ]_{\alpha + \beta -1} \le
    C \| g \|_{\beta}  + C \| b \|_{\beta}  \| Du \|_{0}
    + C \| b \|_{0}
    [ Du ]_{\beta},
$$
we cannot continue, since $\alpha <1$ gives   $Du \in
 C^{\theta}_b$
 with $\theta =  \alpha + \beta -1 < \beta $.
  Roughly speaking,  when $\alpha <1$, the perturbation term
  $b \cdot Du $ is  of order larger than ${\cal L}$
 and  so we cannot prove  the desired
  a-priori estimates.
 }
\end{remark}

\section{The main result}

 We briefly    recall  basic facts about
 Poisson random measures which we use in the sequel (see also
\cite{A}, \cite{Ku}, \cite{protter}, \cite{Zab}).
 The Poisson random measure $N$ associated with
the process $L = (L_t)$ in \eqref{SDE} is defined by
$$
N((0,t] \times U) = \sum_{0 < s \le t} 1_{U} (\triangle L_s) =
 \sharp \{ 0< s \le t \; : \; \triangle L_s \in U\},
$$
 for any Borel set $U$ in $\R^d \setminus \{ 0 \}$, i.e.,
  $U \in {\cal B}(\R^d \setminus \{ 0 \})$,
   $t>0$. Here
 $\triangle L _s = L_s - L_{s-}$ denotes the jump size of $L$
 at time $s > 0.$ The compensated Poisson random measure $\tilde N$ is defined by
 $
 \tilde N ( (0,t] \times U ) =
  N((0,t] \times U) - t \nu (U),
$
 where $\nu $ is given in \eqref{spec}.
  Recall that
  L\'evy-It\^o decomposition of the process $L $
  (see \cite[Theorem 2.4.16]{A} or
 \cite[Theorem 2.7]{Ku}).  This says that
 \begin{equation} \label{ito}
 L_t = \hat b \, t +
 \int_0^t \int_{\{ |x| \le 1\} } x \tilde N(ds, dx) +
 \int_0^t \int_{\{ |x| > 1 \} } x  N(ds, dx), \;\; t \ge 0,
 \end{equation}
where $\hat b= E [ L_1 - \int_0^1 \int_{\{ |x| > 1 \} }
  x  N(ds, dx)]$. Note that in our case, since $\nu$ is symmetric,
  we have $\hat b =0$.

 The stochastic integral
   $\int_0^t \int_{\{ |x| \le 1 \} } x \tilde N(ds, dx)$
 (which is the compensated sum of small jumps)   is an
  $L^2$-martingale.
 The process $ \int_0^t \int_{\{ |x|> 1\} } x N(ds, dx) $
 $= \int_{(0,t]}
\int_{\{ |x|> 1\} } x  N(ds, dx)$ $  { =  \sum_{0 < s \le t,
  \; |\triangle L_s| > 1}
  \triangle L _s}
  $ is   a compound Poisson
process.

 Let $T>0$.
 The predictable $\sigma$-field ${\cal P}$
 on $\Omega \times [0,T]$ is generated by all  left-continuous adapted processes (defined on the same stochastic basis fixed in Section 2).
 Let $U \in {\cal B}(\R^d \setminus \{ 0\}).$ In the sequel, we will always consider   a ${\cal P} \times {\cal B}(U)$-measurable mapping  $F :
 [0, T] \times U \times \Omega \to \R^d$.

If $0 \not \in \bar {U},$ then  $\int_0^T \int_{U} F(s,x)  N(ds, dx) $ $ = \sum_{0 < s \le T
  \; }
  F(s, \triangle L _s) 1_U (\triangle L _s)$ as a random finite sum.

If $E \int_0^T ds \int_U |F(s, x)|^2 \nu (dx) < \infty$, then one can
define
  the stochastic integral
 $$
 Z_t =  \int_0^t \int_{U} F(s,x)  \tilde N(ds, dx), \;\; t \in [0,T]
 $$
 (here we do not assume $0 \not \in \bar {U}$).
  The process
$Z = (Z_t)$ is an $L^2$-martingale with a c\`adl\`ag modification.
Moreover, $E |Z_t|^2$
 $= E\int_0^t ds \int_{U} |F(s,x)|^2  \nu(dx)$
 (see \cite[Lemma 2.4]{Ku}). We will also use the following
 $L^p$-estimates
 (see \cite[Theorem  2.11]{Ku}
 or the proof of  Proposition 6.6.2 in \cite{A});
  for any $p \ge 2$,
 there exists $c(p) >0$ such that
 $$
E [\sup_{0 < s \le t} |Z_s|^p] \le c(p)
   E \Big [ \big(\int_0^t ds \int_{U}
   |F(s,x)|^2  \nu(dx) \big)^{p/2}
 \Big] $$
 \beq \label{kun}+ \, c(p)  E \Big [\int_0^t ds \int_{U}
 |F(s,x)|^p  \nu(dx)
 \Big],\;\; t \in [0,T]
\eeq
 (the inequality is obvious if the right hand side is infinite).

\vskip 1mm  Let us recall the concept of {\it $($strong$)$ solution} which we consider.  A solution
 to the SDE \eqref{SDE} is a c\`adl\`ag ${\cal F}_t$-adapted  process $X^x=
(X_t^x)$ (defined on  $(\Omega, {\cal F}, ({\cal F}_t)_{t \ge 0},
P)$ fixed in Section 2) which solves \eqref{SDE} $P$-a.s., for
 $t
\ge 0$.

\vskip 0.5 mm
 It is easy to show the existence of a solution to
\eqref{SDE}
 using the fact that $b$ is bounded and continuous. We may argue
 at $\omega$ fixed.
 Let us first consider $t \in [0,1]$.
  By introducing $v(t) = X_t - L_t$, we get
the equation
 $$
 v(t) = x + \int_0^t b(v(s) + L_s) ds.
 $$
Approximating $b$ with smooth drifts $b_n$
 we find solutions $v_n \in C([0,1]; \R^d)$.
 By  the Ascoli-Arzela theorem, we  obtain
  a solution to
 \eqref{SDE} on $[0,1]$. The same argument works also on the time
 interval
  $[1,2]$  with a random initial condition. Iterating this
   procedure
   we can
  construct a solution for all $t \ge 0$.

%\vskip 1mm
The proof of Theorem \ref{uno} requires  some lemmas.
We begin with a deterministic result.

\begin{lemma} \label{dis} Let $\gamma \in [0, 1]$
 and  $f \in C^{1+ \gamma}_b(\R^d)$.
  Then for any $u, v \in \R^d$, $x \in \R^d,$ with  $
  |x| \le 1$,   we have
  $$
 | f(u + x) - f(u) - f(v+ x) + f(v) | \le c_{\gamma}
 \| f\|_{{1+ \gamma}}
 \, |u- v| \, |x|^{\gamma}, \;\; \; \text{with} \; c_{\gamma}
 = 3^{1- \gamma} 2^{\gamma}.
$$
\end{lemma}
 \begin{proof}
 For any $x \in \R^d$, $|x| \le 1$,
  define the linear operator
   $T_x : C^1_b (\R^d) \to C^1_b (\R^d)$,
 $$
 T_x f (u) = f(u + x) - f(u),\;\;\; f \in C^1_b (\R^d),\;
  u \in \R^d.
$$
 Since $\| T_x f\|_{0}
 \le \| Df \|_{0} |x| $ and $ \| D ( T_x f )\|_{0}
 \le 2 \| Df \|_{0}$,
 it follows  that $T_x$ is continuous and
  $\| T_x f\|_{1}$ $ \le (2 + |x|) \,  \|f \|_1,$
   $f \in C^1_b (\R^d)$.
 Similarly,
  $T_x$ is  continuous
  from $ C^{2}_b (\R^d)$ into $ C^1_b (\R^d)$ and
$$
\| T_x f\|_{1} \le 2  |x| \,  \|f \|_2, \;\; \; f \in
 C^2_b
(\R^d).
$$
By interpolation theory $
 \Big ( C^{1}_b (\R^d), C^{2}_b (\R^d) \Big)_{\gamma, \infty}
 = C^{1 + \gamma}_b (\R^d),
 $
 see for instance \cite[Chapter 1]{L}; we deduce that, for any
$\gamma \in [0,1]$,
 $T_x$ is continuous from $
 C^{1 + \gamma}_b (\R^d)$ into  $ C^1_b(\R^d)$ (cf.
 \cite[Theorem 1.1.6]{L})
  with operator norm less than or equal to
 $
 (2 + |x| )^{1- \gamma} \; (2 \, |x|)^{\gamma}.
 $

 Since $|x| \le 1$, we obtain that $ \| T_x  f\|_{1} \le
  c_{\gamma} \, |x|^{\gamma} \,  \| f\|_{1+ \gamma}$, $f
   \in C^{1+ \gamma}_b (\R^d)$.
   Now the assertion follows
 noting that, for any $u, v \in \R^d$,
$$
|f(u + x)  - f(u) - f(v+ x) + f(v)| = |
 T_x f(u) - T_x f(v)|
 \le   \| D T_x f \|_0 \, |u-v|.
$$
The proof is complete.
\end{proof}

In the sequel we will consider the following resolvent equation on
$\R^d$
\begin{equation} \label{resolv1}
 \lambda u -
 {\cal L} u  - Du \cdot b = b,
\end{equation}
 where $b : \R^d \to \R^d$ is given in \eqref{SDE},
 the operator ${\cal L}$ in \eqref{gener}
 and
 $\lambda>0$ (the equation must be understood componentwise, i.e.,
   $\lambda u_i -   $
 ${\cal L} u_i  - b \cdot Du_i = b_i$, $i =1, \ldots, d)$.
 The next two results hold
 for SDEs of type \eqref{SDE} when
 $b$ is only continuous and bounded.

 \begin{lemma} \label{due1} Let $\alpha \in (0,2)$
 and $b \in C_b (\R^d, \R^d)$ in  \eqref{SDE}.
 Assume
 that, for some $\lambda >0$, there exists a solution
 $u \in C^{1+ \gamma}_b (\R^d, \R^d)$ to
 \eqref{resolv1} with $\gamma \in [0,1]$, and moreover
 $$
 1+ \gamma > \alpha.
$$
 Let $X = (X_t)$  be a solution of
 \eqref{SDE} starting at $x \in \R^d$.   We have, $P$-a.s.,
 $t \ge 0$,
\beq \label{itok}
 u(X_t) - u(x)
\eeq
$$
=  x- X_t + L_t  + \lambda  \int_0^t u(X_s) ds
 + \int_0^t \int_{\R^d \setminus \{0 \} } [ u(X_{s-} + x) - u(X_{s-})]
   \tilde N(ds, dx).
$$
\end{lemma}
 \begin{proof} First note that the stochastic integral in
 \eqref{itok} is meaningful thanks to the estimate
  \beq \label{mava}
   E  \int_0^t ds \int_{\R^d \setminus \{0 \} } | u(X_{s-} + x) - u(X_{s-})|^2  \nu(dx)
   \eeq
 $$
  \le 4 t\| u\|_0   \int_ { \{ |x| >1 \}}  \nu(dx)
   + 2 t\| u\|_{1}    \int_{ \{ |x| \le 1 \} }
   | x |^{2 } \nu(dx) < \infty.
   $$
   The  assertion is obtained applying   It\^o's formula to $u(X_t)$
  (for more details on It\^o's formula
   see  \cite[Theorem 4.4.7]{A}
   and \cite[Section 2.3]{Ku}).

  A difficulty is that It\^o's formula
   is usually stated for
   smooth functions $f \in C^2 (\R^d)$.
   However, in the
  present situation in which $L$  is a symmetric
   $\alpha$-stable process, using \eqref{tr}, one can show
  that It\^o's formula holds for
  any  $f  \in C^{1+ \gamma}_b(\R^d)$. We
   give a  proof   of
  this fact.

  Let $f  \in C^{1+ \gamma}_b(\R^d) $.
 We assume that $\gamma >0$ (the proof with $\gamma =0$ is similar).
   By mollifying $f$ as in \eqref{conv} we obtain a
  sequence $(f_n) \subset C^{\infty}_b(\R^d)$ such that
  $
 f_n \to f  \;\; \text{in}$ $ \;\; C^{1+ \gamma'}(K),
 $
  for any compact set $K \subset \R^d$ and $0 < \gamma' < \gamma$. Moreover,  $\| f_n\|_{1 + \gamma} \le \| f\|_{1+ \gamma} $,
 $n \ge 1$ .
 Let us fix $t>0$. By It\^o's  formula
  we find, $P$-a.s.,
\begin{align*}
& f_n(X_t) - f_n(x)
\\
& = \int_0^t \int_{ \R^d \setminus \{ 0\}} [ f_n(X_{s-} + x) -
f_n(X_{s-})]
   \, \tilde N(ds, dx)
\\ &
  + \int_0^t ds\int_{\R^d} [ f_n(X_{s-} + x) - f_n(X_{s-})
   - 1_{ \{  |x| \le 1 \} } \, x \cdot D f_n (X_{s-}) ]
   \nu (dx)
\end{align*}
 \begin{equation} \label{ci}
  + \int_0^t b(X_s) \cdot Df_n(X_s) ds.
 \end{equation}
 It is not difficult to pass to the limit as $n \to \infty$;
  we
 show two  arguments which are needed.
  To deal with the integral
 involving $\nu$, one can apply the dominated convergence theorem, thanks to the following estimate similar to \eqref{rit},
$$
 | f_n(X_{s-} + x) - f_n(X_{s-})
   -  \, x \cdot D f_n (X_{s-}) |
 \le c \| Df\|_{\gamma} |x|^{1+  \gamma}, \;\;
 |x| \le 1
 $$
(recall that  $\int_{ \{ |x| \le 1\} }  |x|^{1+  \gamma} \nu (dx) <
\infty$  since $1+ \gamma > \alpha$). In order to pass to the limit
in the stochastic integral
 with respect to $\tilde N$, one  uses the isometry formula
  \beq \label{dtt}
E \Big|  \int_0^t \int_{\R^d \setminus \{ 0\}} [ f_n(X_{s-} + x)
 -f(X_{s-} + x)  -
f_n(X_{s-}) + f(X_{s-})]
   \tilde N(ds, dx) \Big|^2
\eeq
$$
= \int_0^t ds \int_{\{ |x|\le 1 \} } E | f_n(X_{s-} + x)
 -f(X_{s-} + x)  -
f_n(X_{s-}) + f(X_{s-})|^2
   \nu (dx)
$$
$$
+ \int_0^t ds \int_{\{ |x| >  1 \} } E | f_n(X_{s-} + x)
 -f(X_{s-} + x)  -
f_n(X_{s-}) + f(X_{s-})|^2
   \nu (dx).
 $$
 Arguing as in \eqref{mava},
 since $\| f_n\|_{1+ \gamma} \le \| f\|_{1+ \gamma} $, $n \ge 1$, we
 can apply the dominated convergence theorem in \eqref{dtt}.
 Letting $n \to \infty$ in \eqref{dtt}  we
  obtain 0.
  Finally, we can pass to the limit in probability in
   \eqref{ci}
  and obtain  It\^o's formula  when
 $f \in C^{1+ \gamma}_b(\R^d)$.

 Let now $u \in  C^{1+ \gamma}_b(\R^d, \R^d)$ as in the theorem.
 Noting that, for any $i= 1, \ldots, d$,
$$
 {\cal L} u_i(y) =
  \int_{\R^d} [ u_i (y + x)
-  u_i (y)
   - 1_{ \{  |x| \le 1 \} } \, x \cdot D u_i (y) ]
   \nu (dx),\;\; y \in \R^d,
$$
 and using that $u$ solves \eqref{resolv1}, i.e.,
  ${\cal L}u + b \cdot Du = \lambda u - b$, we can replace
  in the It\^o formula  for $u(X_t)$ the   term
$$
\int_0^t  {\cal L} u(X_s)  ds  +
 \int_0^t Du(X_s) b(X_s)  ds
$$$$ = \sum_{i=1}^d\big( \int_0^t  {\cal L} u_i(X_s)  ds \,  +
 \, \int_0^t Du_i(X_s) \cdot b(X_s)  ds \big) e_i
$$
 with
$ - \int_0^t b(X_s) ds + \lambda  \int_0^t u(X_s) ds
 = x- X_t + L_t + \lambda  \int_0^t u(X_s) ds
$
and obtain the assertion.
\end{proof}

 The proof of  Theorem \ref{uno}
 will be a consequence of the following result.

\begin{theorem}
\label{uno1} Let $\alpha \in (0,2)$ and
  $b \in C_b (\R^d, \R^d)$ in  \eqref{SDE}.
 Assume
 that, for some $\lambda >0$,  there exists a solution
 $u= u_{\lambda} \in C^{1+ \gamma}_b (\R^d, \R^d)$
 to the resolvent equation
 \eqref{resolv1} with $\gamma \in [0,1]$, such that
   $c_{\lambda} =
    \| Du_{\lambda} \|_{0}  < 1/3 $.
 Moreover, assume that
 $$
 2 \gamma > \alpha.
$$
Then the SDE (\ref{SDE}), for every
$x\in\mathbb{R}^{d}$, has a unique  solution $(X_{t}^{x})$.

 Moreover,  assertions (i), (ii) and (iii)
   of Theorem \ref{uno} hold.
 \end{theorem}
 \begin{proof} Note that $2 \gamma > \alpha$ implies
 the condition $1 + \gamma > \alpha$ of Lemma \ref{due1}.

  We  provide a direct proof of
   {\sl pathwise uniqueness  and assertion (i).}
  This uses Lemmas \ref{due1} and \ref{dis}
  together with $L^p$-estimates
  for stochastic integrals (see \eqref{kun}).
  Statements (ii) and (iii)
    will be obtained by transforming
     \eqref{SDE} in a form suitable
     for applying the results
  in   \cite[Chapter 3]{Ku}.

 Let us fix $t>0$, $p \ge 2$
 and
 consider two solutions
  $X $  and $Y$  of
 \eqref{SDE} starting at $x$ and $y \in \R^d$ respectively. Note
   that $X_t$ is not in $L^p$ if $p \ge \alpha$
   (compare with \cite[Theorem 3.2]{Ku}) but the difference $X_t - Y_t$ is a bounded process.
  Pathwise uniqueness and \eqref{ciao} (for any $p \ge 1$) follow if we prove
 \beq \label{ciao2}
E[ \sup_{0 \le s \le t} |X_s - Y_s|^p] \le C(t) \,
 |x-y|^p,\;\;\; x,\, y \in \R^d,
  \eeq
 with a  positive constant $C(t)$  independent of $x$ and $y$.
  Indeed in the special case of $x=y$ estimate \eqref{ciao2} gives
 uniqueness of solutions.

 We have from Lemma \ref{due1}, $P$-a.s.,
 \begin{equation} \label{ser}
 X_t - Y_t = [x-y] + [u(x) - u(y)] + [u(Y_t)- u(X_t)]
 \end{equation}
$$
   +
 \int_0^t \int_{\R^d \setminus \{ 0\}} [ u(X_{s-} + x) - u(X_{s-})
  - u(Y_{s-} + x) + u(Y_{s-})]
   \tilde N(ds, dx)
   $$
$$
 + \lambda  \int_0^t   [u(X_s)- u(Y_s)] ds.
$$
Since
 $\| Du\|_0 \le 1/3 $, we have    $
 |u(X_t)- u(Y_t)| $ $\le \frac{1}{3} |X_t  - Y_t|.
$ It follows the estimate
 $
 | X_t - Y_t| $ $\le \frac{3}{2} \Lambda_1(t)$ $ +
 \frac{3}{2}\Lambda_2(t)$ $ + \frac{3}{2}\Lambda_3(t)
 + \frac{3}{2}\Lambda_4,$ where
$$
 \Lambda_1(t) =  \Big|
 \int_0^t \int_{ \{ |x| > 1\} } [ u(X_{s-} + x) - u(X_{s-})
  - u(Y_{s-} + x) + u(Y_{s-})]
   \tilde N(ds, dx) \Big|,
$$
$$
\Lambda_2(t) =  \lambda  \int_0^t   |u(X_s)- u(Y_s)| ds,
$$
$$
\Lambda_3(t) =  \Big| \int_0^t \int_{\{ |x| \le  1\} } [ u(X_{s-} +
x) - u(X_{s-})
 - u(Y_{s-} + x) + u(Y_{s-}) ]
   \tilde N(ds, dx) \Big|,
$$
$ \Lambda_4 = |x-y| + |u(x) - u(y)| \le \frac{4}{3}|x-y| $. Note
that, $P$-a.s.,
$$
\sup_{0 \le s \le t } |X_s - Y_s|^p \le C_p |x-y|^p +
 C_p\sum_{k=1}^3 \sup_{0 \le s \le t } \, \Lambda_k(s)^p.
$$
The main difficulty is to estimate $\Lambda_3(t)$. Let us first
consider the other terms. By the H\"older inequality
$$
\sup_{0 \le s \le t } \Lambda_2(s)^p \le c_1 (p) \, t^{p-1} \int_0^t
\sup_{0 \le s \le r } |X_s - Y_s|^p \, dr.
%\;\; t \in [0,T].
$$
By \eqref{kun} with $U = \{x \in \R^d \; : \;  |x| >1\}$
 we find
$$
E[\sup_{0 \le s \le t } \Lambda_1(s)^p] $$$$ \le c(p)
 E \Big [ \Big ( \int_0^t ds \int_{ \{ |x| > 1 \} } | u(X_{s-} + x)
  - u(Y_{s-} + x) + u(Y_{s-})- u(X_{s-}) |^2 \nu(dx) \Big )^{p/2}
  \Big ]
$$
$$+ \, c(p)
 E  \int_0^t ds \int_{\{ |x| > 1\} } | u(X_{s-} + x)
  - u(Y_{s-} + x) + u(Y_{s-})- u(X_{s-})|^p \nu(dx).
$$
Using $ | u(X_{s-} + x)
  - u(Y_{s-} + x) + u(Y_{s-})- u(X_{s-}) |
  %2
  % \| Du\|_0 |X_{s-}
  %-  Y_{s-}|
  \le \frac{2}{3} |X_{s-}
  -  Y_{s-}|
  $ and the H\"older inequality, we get
$$
E[\sup_{0 \le s \le t } \Lambda_1(s)^p]
 \le C_1(p) \, (1+ t^{p/2
-1}) \, \cdot \, $$$$ \cdot \, \Big(
 \int_{\{ |x|>1 \} } \nu(dx) \,  +  \big(\int_{\{ |x|>1\} }
\nu(dx) \big)^{p/2} \Big)
 \int_0^t E[
\sup_{0 \le s \le r } |X_s - Y_s|^p ] dr.
$$
Let us treat  $\Lambda_3(t)$. This requires
 the condition $2 \gamma >
\alpha$. By using \eqref{kun} with
 $U = \{x \in \R^d \; : \;  |x| \le 1,\; x \not =0\}$
  and also Lemma \ref{dis}, we get
$$
E[\sup_{0 \le s \le t } \Lambda_3(s)^p]
 \le c(p) \| u\|_{1+ \gamma}^p
 \, E \Big [
 \Big ( \int_0^t ds \int_{\{  |x| \le  1 \} } |X_{s} - Y_s|^2
 |x|^{2\gamma}
 \nu(dx) \Big )^{p/2} \Big ]
$$
$$
 + \, c(p) \| u\|_{1+ \gamma}^p
 \, E  \int_0^t ds \int_{\{  |x| \le  1 \} } |X_{s} - Y_s|^p
 |x|^{\gamma p}
 \nu(dx).
$$
We obtain
$$
E[\sup_{0 \le s \le t } \Lambda_3(s)^p]\le C_2 (p) \, (1+ t^{p/2
-1}) \,
 \| u\|^p_{1+ \gamma} \, \cdot
$$$$ \cdot \, \Big(
 \big( \int_{\{ |x|\le 1 \} } |x|^{2 \gamma} \nu(dx) \big)^{p/2} +
  \int_{\{ |x|\le 1 \} } |x|^{\gamma p}
\nu(dx) \Big)
  \, \int_0^t E[\sup_{0 \le s \le r }
|X_s - Y_s|^p] \, dr,
$$
where $\int_{\{ |x| \le  1 \} }  \, |x|^{p \gamma}
   \nu(dx) < +\infty  $,
 since $p \ge 2$ and $2 \gamma > \alpha$.
 Collecting the previous estimates, we arrive at
$$
E[ \sup_{0 \le s \le t} |X_s - Y_s|^p] \le C_p \,
 |x-y|^p \,+\, C_4 (p) \, (1+ t^{p -1})
 \, \int_0^t E[\sup_{0 \le
s \le r }  |X_s - Y_s|^p] \, dr.
$$
Applying the
  Gronwall lemma we obtain \eqref{ciao2}
  with $C(t) = C_p \exp \big( C_4 (p) \, (1+ t^{p -1})\big)$.
  The assertion is proved.

\vskip 1mm  Now we establish the {\sl  homeomorphism
  property (ii)}
 (cf.  \cite[Chapter 3]{Ku}, \cite[Chapter 6]{A} and
 \cite[Section V.10]{protter}).

  First note that,
  since $\|Du \|_0 <1/3 $,   the classical
  Hadamard
  theorem (see \cite[page 330]{protter}) implies that
 the mapping
  $
  \psi : \R^d \to \R^d$, $\psi (x) = x + u(x)$, $x \in \R^d$,
   is a
  $C^1$-diffeomorphism from $\R^d$
  onto $\R^d$. Moreover,
   $D \psi^{-1}$ is bounded on $\R^d$ and
   $\|D \psi_{}^{-1} \|_0
 \le  \frac{1}{1 - c_{\lambda}} < \frac{3}{2}$ thanks to
\beq \label{stim3} D \psi_{}^{-1}(y)  =  [I + Du_{}( \psi_{}^{-1}
(y))]^{-1} = \sum_{k \ge 0} (- Du_{}( \psi_{}^{-1} (y)))^k,
 \; y \in \mathbb{R} ^d.
 \eeq
 Let $r \in (0,1)$  and  introduce  the  SDE
 \begin{align} \label{itt}
 & Y_t =  y + \int_0^t \tilde b ( Y_s) ds
 \\
 & \nonumber
  \int_0^t \int_{ \{ |z| \le r \} } g(Y_{s-}, z)
 \tilde N(ds, dz) +
 \int_0^t \int_{ \{ |z| > r \} } g(Y_{s-}, z)
 N(ds, dz),\; t \ge 0,
\end{align}
 where $\tilde b(y) = \lambda u( \psi^{-1} (y)) -
 \int_{ \{ |z| > r \} } [ u(\psi^{-1} (y) \,  + z)  - u(\psi^{-1} (y))]
   \nu(dz) $ and
 $$
 g(y,z) =  u(\psi^{-1} (y) \,  + z) + z - u(\psi^{-1} (y)),
 \;\;\; y \in \R^d,\; z \in \R^d \setminus \{ 0\}.
  $$
  Note that \eqref{itt}  is a SDE  of the type considered
   in
  \cite[Section 3.5]{Ku}.
  Due to the Lipschitz condition, there exists a
  unique solution $Y^y = (Y_t^y)$ to \eqref{itt}. Moreover,
   using \eqref{itok} and \eqref{ito} with $\hat b =0$,
    it is not difficult to show that
  \beq \label{spero}
 \psi (X_t^x) = Y_t^{\psi (x)}, \;\; x \in \R^d,\; t \ge 0.
\eeq
 Thanks to \eqref{spero} to prove our   assertion,
  it is enough to show
  the homeomorphism property for  $Y_t^y$.
 To this purpose, we will
 apply \cite[Theorem 3.10]{Ku} to equation \eqref{itt}.
 Let us check its assumptions.

   Clearly,   $\tilde b$
  is   Lipschitz continuous and bounded.
  Let us consider \cite[condition (3.22)]{Ku}.
   For any $y\in \R^d$, $z \in \R^d \setminus \{0\}$, $
 |g(y,z)| \le $ $|z| (1+  \| u\|_1)
  $ $\le K(z)$, where $K(z) = \frac{4}{3}|z|
$ (recall that $ \int_{|z| \le 1} |z|^2 \nu (dz)
  < \infty)$; further by Lemma \ref{dis} and
  \eqref{stim3} we have
$$
|g(y,z) - g(y',z)| \le L(z) |y- y'|,
 \;\;  y, \, y' \in \R^d, \;\; \text{where}
 \; L(z)
= C_1
 \| u\|_{{1+ \gamma}} |z|^{\gamma},
$$
 $|z|\le 1,$ with
  $\int_{|z| \le 1} L(z)^{2} \nu (dz)
  < \infty$,
   since $2 \gamma > \alpha$.
  Note that we may fix $r>0$
  small enough in \eqref{itt}
  in order  that
  $K(r) + L(r)< 1$ (according to
   \cite[Section 3.5]{Ku} this condition  allows to deduce
    that equation  \eqref{itt}
  without $\int_0^t \int_{ \{ |z| > r \} } g(Y_{s-}, z)
 N(ds, dz)$ satisfies the homeomorphism property).

  In order to get the homeomorphism
  property, it remains to check that, for any $z \in \R^d \setminus \{0\}$, the mapping:
\beq \label{hom}
 y \mapsto y + g(y,z) \;\; \text{is a
  homeomorphism from} \; \R^d \; \text{onto} \;
  \R^d.
\eeq
 Let us fix $z$. To check the assertion, we will again apply
  the Hadamard theorem. We have
 $$
 D_y g(y,z) =
 [Du(\psi^{-1} (y) \,  + z)  - Du(\psi^{-1} (y))] \,
  [D \psi^{-1}
(y)]
 $$
 and so by \eqref{stim3} (since $\| Du\|_0 < 1/3$)
 we get $
  \|D_y g( \cdot ,z) \|_0 \le \frac{2 c_{\lambda}}{
  1 - c_{\lambda}} < 1
 $.  We have obtained \eqref{hom}.
 By \cite[Theorem 3.10]{Ku}   the
  homeomorphism property for $Y_t^y$ follows and this
   gives  the assertion.

\vskip 1mm  Now we show that, for any $t \ge 0$,
 the {\sl mapping: $x
 \mapsto X_t^x$ is of class
 $C^1$ on $\R^d$,}  $P$-a.s. (see (iii)).
% {\sl continuously
% differentiable} on $\R^d$,
% $P$-a.s.. This gives (iii).

 We fix $t>0$
 and a unitary vector  $e_k$
 of the canonical basis in $\R^d$.
 % Let $X^x$ be the solution starting at $x$.
% According to  \cite[Section 3.3]{Ku}
 We will  show  that there exists,
 $P$-a.s.,
 the partial derivative $
 \lim_{\lambda \to 0} \frac{X_t^{x+ \lambda e_k}
 - X_t^{x}} {\lambda} $ $ = D_{e_k} X^x_t
 $ and, moreover, that the mapping
  $x\mapsto D_{e_k} X^x_t$ is continuous
 on $\R^d$, $P$-a.s..

  Let us consider the process
 $Y^y= (Y_t^y)$ which solves the SDE
 \eqref{itt}.
  If we prove that
 the mapping
 $y \mapsto Y_t^y$ is of class $C^1$
  on $\R^d$, $P$-a.s.,  then we have   proved the assertion.
 Indeed, $P$-a.s.
 $$ D_{e_k} X^x_t = [ D\psi^{-1} (Y_t^{\psi(x)}) ]
  [DY_t^{\psi(x)}] \, D_{e_k} \psi(x),\;\; x \in \R^d.
 $$
 In order to apply \cite[Theorem 3.4]{Ku}
  we introduce  the process
  $(Z_t^y)$ which solves
   \beq \label{itt2}
 Z_t =  y + \lambda  \int_0^t u( \psi^{-1} (Z_s)) ds
 +
 \int_0^t \int_{  \R^d \setminus \{ 0\}  }
 h(Z_{s-}, z)
 \tilde N(ds, dz),
\eeq $t \ge 0, \; y \in \R^d,$ where
 $$
 h(y,z) =
 u(\psi^{-1} (y) \,  + z)  - u(\psi^{-1} (y))
  = g(y,z) - z
  $$
 (adding $L_t$ to \eqref{itt2},
  one gets \eqref{itt}).
    Proving that   $y \mapsto Z_t^y$
   is of class $C^1$ on $\R^d$,  $P$-a.s.,
 is equivalent to show that $y \mapsto Y_t^y$ is of class $C^1$
 on $\R^d$.
    Indeed, we have
 $
 \lim_{\lambda \to 0} \frac{Y_t^{y+ \lambda e_k}
 - Y_t^{y}} {\lambda}
 $ $ =  \lim_{\lambda \to 0} \frac{Z_t^{y+ \lambda e_k}
 - Z_t^{y}} {\lambda}.
$

 To prove the assertion for $Z_t^y$, it is enough to
 check that the SDE \eqref{itt2} verifies the assumptions
 of \cite[Theorem 3.4]{Ku}.
 These are, respectively,
  \cite[conditions (3.1), (3.2), (3.8) and (3.9)]{Ku}.
 Conditions (3.1) and (3.2) are easy to check. Indeed
 $\lambda u
(\psi^{-1} (\cdot))$ is Lipschitz continuous on $\R^d$ and,
 moreover, thanks to Lemma \ref{dis} and to the boundeness of
 $D \psi^{-1}$,
$$
|h(y, z) - h(y', z)| \le C ( 1_{\{ |z|\le 1 \}} |z|^{\gamma} +
 1_{\{ |z|> 1 \}}) \; |y-y'|, \;\; z \in \R^d \setminus \{ 0\},
$$
 $y, y'  \in \R^d$,
   with $\int_{\R^d } ( 1_{\{ |z|\le 1 \}}
 |z|^{\gamma} +
 1_{\{ |z|> 1 \}})^p \, \nu (dz) < \infty$, for any $p \ge 2$.
 In addition,
$|h(y, z) | \le L(z)$, $z \in \R^d \setminus \{ 0\},
 $
 $y \in \R^d$, where, since $\| Du\|_0 < 1/3,$
$$
L (z) =   \frac{1}{3} \,  1_{\{ |z|\le 1 \}} |z|^{} +
 2 \| u\|_0 1_{\{ |z|> 1 \}}\;\;\; \text{with}\;\;
\int_{\R^d } L(z)^p \nu (dz) < \infty, \;\; p \ge 2.
$$
 Assumptions \cite[(3.8) and (3.9)]{Ku} are more difficult
  to check.
 They require that there exists some $\delta >0$ such that
  (setting $l(x) =  \lambda u
(\psi^{-1} (x))$)
 %\beq\label{ku2}
$$
(1) \; \sup_{y \in \R^d} | D l (y))| < \infty;
 \;\; |Dl(y) - Dl(y') |
  \le C |y-y'|^{\delta},\;\; y,\, y' \in \R^d.
 $$
\beq \label{ch} (2) \; | D_y h (y,z))| \le K_1(z);
 \;\; |D_y h(y,z) - D_y h(y',z) | \le K_2(z) \,
 |y-y'|^{\delta},  \;
\eeq for any  $y,\, y' \in \R^d,$
 $ z \in \R^d \setminus \{ 0\},$
with
 $\int_{\R^d  } K_i(z)^p  \, \nu (dz) < \infty$,
 for any $p \ge 2$, $i=1,2$. Such estimates are used in
 \cite{Ku} in combination with the
  Kolmogorov continuity theorem
 to show the differentiability property.

 Let us check (1) with $\delta = \gamma$, i.e.,  $Dl \in
   C_b^{\gamma} (\R^d, \R^d)$.
Since,  for any $y \in \R^d$,
 $
D  l (y) = \lambda Du ( \psi^{-1} (y))
 D\psi^{-1} (y)
$, we find that $Dl $ is bounded on $\R^d$.
 Moreover, thanks to the following
 estimate (cf. \eqref{pri})
$$
 [Dl]_{\gamma} \le \lambda \| Du \|_0 [D \psi^{-1}]_{\gamma}
 \, +
 \, \lambda [ Du ]_{\gamma} \| D \psi^{-1}\|_{0}^{1+ \gamma},
$$
 in order  to prove the assertion it is enough to show that
  $[D \psi^{-1}]_{\gamma} < \infty$.
 Recall that for  $d \times d $
 real matrices $A $ and $ B$,   we have
 $(I + A)^{-1} - (I + B)^{-1} $ $=
  (I + A)^{-1} ( B -A ) (I + B)^{-1}$
 (if $(I + A)$ and $(I + B)$ are invertible).
 We obtain, using also  that  $D \psi_{}^{-1}$ is bounded,
$$
|D \psi^{-1}(y) - D \psi^{-1}(y')| =
 |[I + Du_{}(
\psi_{}^{-1} (y))]^{-1} - [I + Du_{}( \psi_{}^{-1} (y'))]^{-1}
 |
$$
$$
\le c_1  \, [ Du ]_{\gamma} \,
 |y- y'|^{\gamma}, \;\; y, \, y' \in
\R^d
$$
 and the proof of (1) is complete
  with $\gamma = \delta$.
   Let us consider (2). Clearly,
 $ D_y h(y,z) = [Du(\psi^{-1} (y) \,  + z)  -
 Du(\psi^{-1} (y))]  D \psi^{-1}(y)$ verifies the first part of
 (2) with
$$
 K_1 (z) = c_2 \| Du\|_{\gamma} \,
  ( 1_{\{ |z|\le 1 \}} |z|^{\gamma} +
 1_{\{ |z|> 1 \}}).
$$
Let us deal with  the second part of (2). We choose
  $\gamma' \in (0, \gamma)$ such that $2 \gamma'
  > \alpha$ and
   first show that,
  for any $f \in C^{\gamma}_b (\R^d, \R^d)$, we have
\beq \label{inter}
 [T_x f]_{\gamma - \gamma'} \le C [f]_{\gamma} \,
 |x|^{\gamma'}, \;\;
  x \in \R^d,
\eeq
 where (as in Lemma \ref{dis}) for any $x \in \R^d$,
  we define
   the mapping $T_x f: \R^d \to \R^d$ as
  $T_x f (u) = f(x+ u ) - f(u)$, $u \in \R^d$.
  Using also \eqref{kry11} we get
$$
[T_x f]_{\gamma - \gamma'} \le N [T_x f]_{\gamma}^{\frac{\gamma -
\gamma'}{\gamma}} \, [T_x f]_{0}^{1 - \, \frac{\gamma -
\gamma'}{\gamma}}  \le c N [f]_{\gamma} \, |x|^{\gamma (1 - \,
\frac{\gamma - \gamma'}{\gamma})} \le cN |x|^{\gamma'} [f]_{\gamma},
$$
 for any $x \in \R^d$. By \eqref{inter} we will prove (2)
 with $\delta =
 \gamma - \gamma' >0$.

 First consider
 the case when $|z| \le 1$. By \eqref{inter}
  with $Du = f$, we get
$$
|D_y h(y,z) - D_y h(y',z) | $$$$ = |Du(\psi^{-1} (y) \,  + z)  - D
u(\psi^{-1} (y))
 - Du(\psi^{-1} (y') \,  + z)  + D
u(\psi^{-1} (y')) | \, \|D \psi^{-1}\|_0
$$
$$
\le C_1 [Du]_{\gamma} |y - y'|^{\delta} \, |z|^{\gamma'},
$$
for any $y, y' \in \R^d$. Let now $|z|>1$; we find, for
 $y, y' \in \R^d$ with $|y- y' | \le 1$,
$$
|D_y h(y,z) - D_y h(y',z) | \le
C_2 [Du]_{\gamma} |y - y'|^{\gamma}
 \le C_2 [Du]_{\gamma} |y - y'|^{\gamma - \gamma'}.
$$
On the other hand,  if $|y- y' | > 1$, $|z| >1$,
$$
|D_y h(y,z) - D_y h(y',z) | \le 4 \| Du\|_{0} |y - y'|^{\gamma-
\gamma'}.
$$
In conclusion, the second part of (2) is verified with
 $\delta = \gamma- \gamma'$ and
$$
K_2 (z) = C_3 \| Du\|_{\gamma} \,
 ( 1_{\{ |z|\le 1 \}} |z|^{\gamma'}
+
 1_{\{ |z|> 1 \}}).
$$
 (note that $\int_{\R^d  } K_2(z)^p  \, \nu (dz) < \infty$,
 for any $p \ge 2$, since $2 \gamma' > \alpha$).
 Since $C_{b}^{\gamma}\left(  \mathbb{R} ^{d} , \R^d \right)
  \subset C_{b}^{\gamma - \gamma'}
  \left(  \mathbb{R} ^{d} , \R^d \right) $,
 we deduce that both (1) and (2) hold with
 $\delta = \gamma - \gamma'.$

Applying \cite[Theorem 3.4]{Ku}, we get that
  $ y \mapsto Z_t^y$ is $C^1$, $P$-a.s., and this proves our
  assertion.
 We finally note that  \cite[Theorem 3.4]{Ku} also provides
  a formula for  $H_t^y = D Z_t^y = DY_t^y$, i.e.,
  \begin{align*}
 H_t^y & =  I + \lambda  \int_0^t Du( \psi^{-1} (Z_s^y))
  \, D\psi^{-1}(Z_s^y) \, H_s^y \, ds
 \\ & + \int_0^t \int_{  \R^d \setminus \{ 0\}  }
  \Big ( D_y h(Z_{s-}^y, z) \, H_{s-}^y  \Big) \,
 \, \tilde N(ds, dz), \;\; t \ge 0, \; y \in \R^d.
\end{align*}
 The stochastic integral is meaningful, thanks to
 (2) in \eqref{ch} and also to
 the fact that  \cite[assertion (3.10)]{Ku}
  implies that, for any $t>0$, $p \ge 2$,
  $ \sup_{0 \le s \le t} E[ |H_s|^p ] $ $< \infty$.
 The proof is complete.
 \end{proof}

 \noindent \textit{Proof of Theorem \ref{uno}.}
   We may assume that  $1 - \alpha/2 < \beta <  2 - \alpha$.
  We will deduce the assertion from Theorem \ref{uno1}.

  Since
  $\alpha \ge 1$, we can apply
   Theorem \ref{reg} and find
   a solution $u_{\lambda}
   \in C^{1+ \gamma}_b (\R^d, \R^d)$
   to the resolvent equation \eqref{resolv1} with
    $\gamma = \alpha - 1 + \beta \in (0,1)$.
    By
  the last assertion of Theorem \ref{reg}, we may
    choose $\lambda$ sufficiently large in order
    that $\| Du\|_0 = \| Du_{\lambda}\|_0 < 1/3$.
   The crucial assumption about $\gamma$ and $\alpha$ in
 Theorem \ref{uno1} is satisfied. Indeed
 $
 2 \gamma = 2\alpha - 2 + 2 \beta > \alpha
 $ since $\beta > 1 - \alpha/2$.
 By  Theorem \ref{uno1} we obtain the result.
 \qed

%\vskip 2mm
%As a consequence of \eqref{ciao} in  Theorem \ref{uno}, using the Kolmogorov
%continuity theorem (see  \cite[page 218-220]{protter}
% and \cite[Section 3.2]{Ku}), we obtain

\begin{remark}\label{flow1} {\em Thanks  to Theorem
\ref{uno} we may define a stochastic flow associated to \eqref{SDE}.
  To this purpose, note that by  (ii) we have
  $X_t^x =\xi_t (x)$, $t \ge 0,$ $x \in \R^d$,
  $P$-a.s.., where $\xi_t$
  is a homeomorphism from $\R^d$ onto $\R^d$.
   Let $\xi_t^{-1}$ be the inverse map.
  As in \cite[Section 3.4]{Ku}, we set
  $
  \xi_{s,t}(x) = \xi_t \circ \xi_s^{-1}(x),\;\; 0 \le
  s \le t, \; x \in \R^d.
  $

 The family $(\xi_{s,t})$ is a stochastic flow since verifies the
 following properties ($P$-a.s):
 (i) for any $x \in \R^d$, $(\xi_{s,t}(x))$
 is a c\`adl\`ag  process with respect to $t$ and a
  c\`adl\`ag  process with respect $s$; (ii) $\xi_{s,t}:
  \R^d \to \R^d$ is an  onto homeomorphism, $s \le t$; (iii)
  $\xi_{s,t}(x)$ is  the unique solution to \eqref{SDE}
  starting from $x$ at time $s$; (iv)
   we have $\xi_{s,t}(x)
=\xi_{u,t}(\xi_{s,u}(x))$, for all $0\leq s\leq u\leq t $,
$x\in{\mathbb{R}}^{d}$, and $\xi_{s,s}(x)=x$.}
\end{remark}

%%%
\def\cii{
\begin{remark}\label{lip} {\em Pathwise uniqueness and
   assertions
  (i) and (ii) of Theorem \ref{uno}
 can be proved  for more general  SDEs like
 $dX_t = (b(X_t ) + F(X_t))dt + dL_t$
  with $F : \R^d \to \R^d$ which is Lipschitz continuous.
 The proof is similar to the one of
    Theorem \ref{uno1}.
  Indeed with the term $F(X_t) dt$, we would have in
   the right-hand side of formula
  \eqref{itok} also the regular term
$$
   \int_0^t Du(X_s) F(X_s) ds,
$$
 where $u$ solves $\lambda u  - {\cal L}u -
  b \cdot Du = b$. The new term
   does not create any additional difficulty in proving (i)
 and (ii).
%Also the initial condition $x \in \R^d$ can be replaced by
%  any ${\cal F}_0$-measurable random variable $\xi$.

 }
\end{remark}
}
 %%%%%%%%%

\vskip  3mm \noindent   \textit{Acknowledgements.}
  The author is grateful to
 F. Flandoli for drawing his attention to
 a  question by L. Mytnik
  which was the starting point of this work. He also
  thanks the Newton Institute (Cambridge) for good working conditions.

\vspace{ 5 mm } {\small
 Dipartimento di Matematica, Universit\`a
di Torino \par
  via Carlo Alberto 10  \ 10123
  \par  Torino, Italy \par
 e-mail: enrico.priola@unito.it }
\par \ \par

\end{document}